\documentclass[12pt]{amsart}

\usepackage{float}
\usepackage{listings}
\usepackage{lmodern,amsmath,amssymb,amsthm,mathtools,mathrsfs,booktabs,longtable}
\mathtoolsset{showonlyrefs=true}

\usepackage[a4paper,left=2cm,right=2cm,top=1cm,bottom=1cm]{geometry}

\usepackage{cite}
\usepackage{microtype}
\usepackage{needspace}
\usepackage[colorlinks=true,linkcolor=blue,citecolor=blue,urlcolor=blue]{hyperref}
\allowdisplaybreaks
\newtheorem{theorem}{Theorem}[section]
\newtheorem{proposition}[theorem]{Proposition}
\newtheorem{lemma}[theorem]{Lemma}
\newtheorem{corollary}[theorem]{Corollary}

\theoremstyle{definition}
\newtheorem{remark}[theorem]{Remark}
\newcommand{\F}{\mathbb F_2}
\newcommand{\A}{\mathcal A}
\newcommand{\Sq}{\operatorname{Sq}}
\newcommand{\Hit}{\mathcal H}
\newcommand{\Crel}{\mathcal C}
\newcommand{\Rel}{\mathcal R}
\newcommand{\DS}{\operatorname{DS}}
\newcommand{\DP}{\operatorname{DP}}
\newcommand{\Stab}{\operatorname{Stab}}
\newcommand{\GL}{\operatorname{GL}}
\newcommand{\Ext}{\operatorname{Ext}}
\newcommand{\im}{\operatorname{im}}
\newcommand{\rank}{\operatorname{rank}}
\newcommand{\tr}{\operatorname{tr}}
\newcommand{\Orb}{\mathscr S}
\newcommand{\e}{\mathrm e}
\DeclareMathOperator{\SSF}{SSF}
\DeclareMathOperator{\SF}{SF}

\title[Relative cohits in the symmetric hit problem]{The symmetric hit problem in four variables:\\ Relative cohits and stabilizer-fixed preimages}

\author[\DJ. V. Ph\'uc]{\DJ{}\d{\u{a}}ng V\~o Ph\'uc$^{1,\, *}$}
\address{$^{1}$Department of Mathematics, FPT University, An Phu Thinh New Urban Area, Quy Nhon, Vietnam}
\address{$^{1}$Department of Mathematics and Statistics, Quy Nhon University, 170 An Duong Vuong, Quy Nhon, Vietnam}
\email{dangphuc150488@gmail.com}
\thanks{$^{*}$Corresponding author}
\thanks{ORCID (\DJ. V. Ph\'uc): \url{https://orcid.org/0000-0002-6885-3996}}

\author[P. V. Danchev]{Peter V. Danchev$^{2}$}
\address{$^{2}$Institute of Mathematics and Informatics, Bulgarian Academy of Sciences, Bulgaria}
\email{danchev@math.bas.bg; pvdanchev@yahoo.com}
\thanks{ORCID (P.V. Danchev): \url{https://orcid.org/0000-0002-2016-2336}}

\date{}

\begin{document}

\begin{abstract}
In this work, we study the symmetric hit problem in four variables by analyzing the relative cohit space, defined as the kernel of the inclusion-induced map from symmetric to ordinary cohits. We establish that this relative space is isomorphic to a kernel in the first group cohomology of the permutation group and, equivalently, to an \(\operatorname{Ext}^1\) kernel for the general linear group, equipped with a canonical action of the normalizer quotient. We use explicit Lucas--Cartan identities to construct unconditional fixed Steenrod preimages for nontrivial stabilizers, resolving infinite families of monomials strictly above the minimal-spike weight. Using the ordinary and symmetric Kameko isomorphisms, we reduce the determination of complete bases in two generic degree families to exact calculations in degrees one, three, six, ten, and sixteen. The relative space first becomes nonzero in degree fifteen, where it is one-dimensional, and it is an indecomposable two-dimensional normalizer module in degree eighteen. In the family \(5\cdot2^t-4\), the symmetric cohit dimension stabilizes at six and the relative dimension at one for \(t\geq2\); in the family \(7\cdot2^t-4\), they are four and zero for \(t\geq1\). We give complete bases by explicit exponent formulas and prove Janfada's monomial implication \cite[Conjecture~1]{Janfada2011} in every degree of both families. We extend the known optimal bound for ordinary Kameko isomorphisms to every permutation-invariant cohit space, prove its optimality there, and transport the entire symmetric hit basis by a direct-sum formula.
\end{abstract}

\maketitle

\section{Introduction}

The symmetric hit problem originates from the profound interplay between the mod \(2\) Steenrod algebra and the topology of real orthogonal classifying spaces. By virtue of the splitting principle, the Stiefel--Whitney characteristic classes are identified with the elementary symmetric polynomials in degree-one roots. This canonical identification equates the Steenrod module \(H^*(BO(n);\F)\) with the symmetric polynomial algebra in \(n\) variables \cite{MilnorStasheff,JanfadaWood2003,Singer2008}. Consequently, determining a minimal homogeneous generating set for this cohomology as a module over the Steenrod algebra reduces to determining which symmetric polynomials lie in the span of the images of positive Steenrod operations. However, a fundamental barrier arises: an ordinary Steenrod expression for a symmetric polynomial does not guarantee that the Steenrod preimages themselves are symmetric. Resolving this equivariant discrepancy is the central theme of the symmetric hit problem.

Historically, the ordinary Peterson hit problem has seen fundamental reductions, most notably Wood's proof of the Peterson conjecture and Singer's minimal-spike criterion \cite{Wood1989,Singer1991}. The ordinary three-variable problem was determined by Kameko \cite{Kameko1998}. In the invariant setting, Janfada and Wood laid the groundwork by establishing the symmetric Peterson theorem and determining the complete set of generators in three variables \cite{JanfadaWood2002,JanfadaWood2003}. Janfada's subsequent investigations address monomial criteria, instability, and symmetrization \cite{Janfada2008,Janfada2009Sym,Janfada2009Unstability,JanfadaCriteriaB3}. His monomial conjecture asserts that the orbit sum of every individually hit monomial is symmetrically hit \cite[Conjecture~1]{Janfada2011}; he proved it for monomials with distinct exponents \cite[Theorem~3]{Janfada2011}. Pengelley and Williams extended symmetric sparseness to all primes; its mod-\(2\) case recovers the Janfada--Wood result \cite{PengelleyWilliams2015}. In their monographs, Walker and Wood develop these topics using binary blocks, the divided-power dual, modular representations, and Kameko maps \cite{WalkerWoodI,WalkerWoodII}. The ordinary four-variable cohit dimensions were determined by Sum \cite[Theorem~4]{Sum2013}; these known dimensions provide independent checks for our calculations in this paper. The optimal bound for ordinary Kameko isomorphisms is due to T\'in and Sum \cite[Theorem~3]{TinSum2016}. Despite these advancements, the case of individually hit monomials with repeated exponents and weight not strictly below the minimal-spike weight remains unresolved in general. For such a monomial, summing an ordinary preimage equation over the full permutation group multiplies its orbit sum by the even order of its stabilizer, giving zero in characteristic two. The coset norm produces the desired orbit sum, but its application to the preimages requires them to be fixed by that stabilizer.

In this paper, we resolve the aforementioned equivariant obstruction for explicit infinite families in four variables by systematically studying the relative cohit space, which measures the failure of the inclusion of symmetric polynomials to induce an injection on cohits. Specifically, we address the fixed-preimage obstruction and present the following main theoretical and computational results:
\begin{itemize}
    \item \textit{Cohomological Duality of the Relative Space:}
    We identify the dual of the relative cohit space with a quotient of the symmetric Steenrod kernel by restrictions of ordinary kernel functionals. Furthermore, an equivariant exact sequence identifies the obstruction to fixed Steenrod preimages with a class in first group cohomology, represented by an explicit one-cocycle. The relative cohit space is naturally isomorphic to a kernel of maps between extension groups \(\Ext^1\) for \(\GL(n,\F)\) (Theorem~\ref{thm:relative-structure}).

\medskip

    \item \textit{Exact Fixed-Preimage Constructions:}
    For monomials with repeated exponents, we establish explicit Lucas--Cartan identities that yield preimages fixed by nontrivial stabilizers, including \(\Sigma_2\) and \(\Sigma_2\times\Sigma_2\). These uniform constructions resolve the fixed-preimage condition for infinite families strictly above the minimal-spike weight, bypassing the need for a local descent hypothesis (Theorems~\ref{thm:unconditional-pair} and~\ref{thm:two-pair-family}).

\medskip

    \item \textit{Relative Nonvanishing and the Normalizer Action:}
    We prove that the first nonzero relative cohit space in four variables occurs in degree \(15\), where it is one-dimensional and generated by \([m_{7,5,3,0}]\) (Theorem~\ref{thm:fifteen}). Since the ordered monomials in its orbit are non-hit, it does not provide a counterexample to Janfada's monomial implication \cite[Conjecture~1]{Janfada2011}. The distinction is analogous to the example \(m_{8,7,1}\) in three variables recorded by Walker and Wood \cite[Example~25.2.2]{WalkerWoodII}. Furthermore, in degree \(18\), the relative space is the indecomposable two-dimensional regular module \(\F C_2\) for the normalizer quotient (Theorem~\ref{thm:normalizer-eighteen}).

\medskip

    \item \textit{Finite Reduction and Complete Bases in Generic Degrees:}
    The classical and symmetric Kameko isomorphisms reduce the two families
    \(5\cdot2^t-4\) and \(7\cdot2^t-4\) to five finite initial degrees.
    We determine explicit bases of symmetric cohits and transport the
    complete symmetric hit bases by a direct-sum formula
    (Proposition~\ref{prop:transport-complete-bases}).
    In the first family the symmetric cohit and relative dimensions are
    \(6\) and \(1\) for \(t\geq2\); in the second they are \(4\) and \(0\)
    for \(t\geq1\). Theorems~\ref{thm:five-generic-family}
    and~\ref{thm:seven-generic-family} give the bases and establish
    Janfada's implication \cite[Conjecture~1]{Janfada2011} for every
    monomial in both families. The finite reduction uses the optimal
    bound of T\'in and Sum \cite[Theorem~3]{TinSum2016}. We prove that
    the same bound is optimal for the cohits of every
    permutation-invariant algebra, including the symmetric algebra
    (Proposition~\ref{prop:eventual-kameko-stability}). In four variables
    the required initial index is at most two. An elementary
    classification of ordinary hit monomials in every power-of-two
    degree gives a direct proof of the degree-sixteen monomial test
    (Proposition~\ref{prop:power-of-two-monomials}).
\end{itemize}

\vspace{2mm}
\noindent\textbf{Methodology.}
Our approach requires a synthesis of group cohomology, algebraic combinatorics, and exact finite-field computation. The normalizer quotient of the permutation subgroup is cyclic of order two in four variables and acts canonically on the relative cohit space. For the combinatorial construction of symmetric hit identities, we use an orbit-coefficient parity formula: after taking the norm of a fixed-preimage equation, the coefficient of a target orbit is the index of the stabilizer intersection in the target stabilizer, reduced modulo two. We couple this parity analysis with digitwise Lucas conditions to force prescribed terms in exact Cartan expansions. The finite-degree assertions are verified by exact Gaussian elimination over \(\F\). Appendices~\ref{app:finite-verifier} and~\ref{app:fixed-verifier} contain complete reference programs that reconstruct the matrices, produce the fixed-preimage expressions, and check the resulting identities using integer binomial coefficients before reduction modulo two.

\vspace{2mm}
\noindent\textbf{Paper Organization.}
The remainder of the paper is structured to separate the general cohomological theory from the specific algebraic combinatorics of four variables. Section~\ref{sec:prelim} fixes the standard notation for Steenrod modules, binary blocks, and dual spaces. Section~\ref{sec:relative} establishes the theoretical framework, proving the equivariant exact sequence, the duality theorems, and the representation-theoretic connections via \(\GL(n,\F)\). In Section~\ref{sec:stabilizers}, we shift to the structural combinatorics, constructing fixed preimages via the stabilizer parity formula and formulating the explicit descent criterion for repeated exponents. Section~\ref{sec:certificates} is devoted to the nonzero relative classes, proving the nonvanishing relative class in degree \(15\) and the indecomposable normalizer module in degree \(18\). Section~\ref{sec:kameko} first establishes the compatibility of relative cohits with Kameko maps, expresses the optimal bound for successive Kameko isomorphisms in terms of binary digit sums, and proves the transport of complete bases. Finally, Section~\ref{sec:computation} applies these results to the two generic degree families: it gives the finite initial calculations, their explicit cohit bases and Cartan certificates, and the resulting formulas in every degree of each family.

\section{Preliminaries}\label{sec:prelim}

The arguments use the Steenrod action, permutation invariants, and binary
weights simultaneously. We fix these conventions first, then introduce
the dual spaces, the equivariant preimage sequence, and the finite matrices.
All subsequent constructions use the notation established here.

\subsection{Steenrod modules and orbit sums}

The field is \(\F\), the mod \(2\) Steenrod algebra is \(\A\), and its
augmentation ideal is \(\A^+\). Let
\[
 P(n)=\F[x_1,\ldots,x_n],\quad |x_i|=1,\qquad
 G_n=\GL(n,\F),\quad H_n=\Sigma_n,\quad B(n)=P(n)^{H_n}.
\]
We use right actions: \(x_i\cdot g=\sum_jg_{ij}x_j\), extended
multiplicatively. The permutation matrices form \(H_n\).
The elementary symmetric polynomial of degree \(j\) is \(\e_j\).
For \(a\in\mathbb N^n\), where \(\mathbb N=\{0,1,\ldots\}\), write
\(x^a=\prod_i x_i^{a_i}\) and \(|a|=\sum_i a_i\).
The Steenrod action is determined by
\(\Sq(x_i)=x_i+x_i^2\). The binomial action and the Cartan formula
\cite[Propositions~1.1.8--1.1.9]{WalkerWoodI} give
\begin{equation}\label{eq:cartan}
 \Sq^r(x^a)=
 \sum_{\substack{u\in\mathbb N^n\\|u|=r}}
 \left(\prod_{i=1}^n\binom{a_i}{u_i}\right)x^{a+u}.
\end{equation}
All coefficients are interpreted in \(\F\).
For an integer \(a\geq0\), \(\alpha(a)\) is its binary digit sum.
The notation \(u\subseteq_2 a\) means that every nonzero binary digit
of \(u\) is a nonzero digit of \(a\).

For a graded Steenrod module \(M\), put
\[
 \Hit^d(M)=\A^+M\cap M^d,\qquad Q_d(M)=M^d/\Hit^d(M).
\]
Every positive word in squares has a positive leftmost square. Hence
\[
 \Hit^d(M)=\sum_{r>0}\Sq^r(M^{d-r}).
\]
For the polynomial modules used here, instability makes the summands
with \(r>d/2\) zero \cite[Proposition~1.1.10]{WalkerWoodI}.

For subgroups \(H\leq G\leq G_n\), an \(H\)-fixed polynomial \(f\)
has norm
\[
 N_H^G(f)=\sum_{\eta\in H\backslash G}f\cdot\eta.
\]
For a monomial \(x^a\), its stabilizer in \(H_n\) is \(H_a\), and
\[
 m_a=\sigma(x^a)=N_{H_a}^{H_n}(x^a).
\]
Thus \(m_a\) depends only on the nonincreasing rearrangement
\(\lambda(a)\) of \(a\). The \(m_\lambda\) indexed by partitions of \(d\)
of length at most \(n\) form a basis of \(B^d(n)\).
The linear map \(\Orb:P^d(4)\to B^d(4)\) sends each ordered
monomial \(x^a\) to \(m_a\). This map is not assumed to commute with
Steenrod squares.

\subsection{Blocks, spikes, and descent}

Write \(a_i=\sum_{j\geq1}a_{ij}2^{j-1}\).
The binary block of \(a\) has entries \(a_{ij}\), with the lowest binary
digit first. Its weight sequence is
\[
 \omega(a)=(\omega_1(a),\omega_2(a),\ldots),
 \qquad \omega_j(a)=\sum_i a_{ij}.
\]
We compare weight sequences in left lexicographic order, denoted \(<_l\).
A spike has exponents \(2^r-1\), allowing \(r=0\).
For \(d>0\), \(\mu(d)\) is the least number of positive spike exponents
whose sum is \(d\), and \(\mu(0)=0\).
For \(d>0\), the minimal spike partition is the unique partition
of \(d\) into positive numbers \(2^r-1\) whose parts are distinct,
except that the smallest part may occur twice. Equivalently, its
positive spike lengths satisfy
\(r_1>\cdots>r_{s-1}\geq r_s>0\) when \(s\geq2\).
It has \(s=\mu(d)\) parts and is obtained by successively
subtracting the largest possible spike exponent
\cite[Theorem~5.4.2 and Definition~5.4.3]{WalkerWoodI}.
When \(\mu(d)\leq n\), a minimal spike in \(P(n)\) has these
positive exponents and \(n-\mu(d)\) zero exponents; its weight
is \(\omega_{\min}(d)\). For \(d=0\), the minimal spike is
\(1\), with zero weight sequence. For every integer \(j\geq0\),
the numerical criterion is
\begin{equation}\label{eq:mu}
 \mu(d)\leq j\quad\Longleftrightarrow\quad\alpha(d+j)\leq j;
\end{equation}
see \cite[Proposition~2.4.4]{WalkerWoodI}.
For a four-tuple \(a\), let \(\alpha^*(a)\) be the nondecreasing
rearrangement of \((\alpha(a_1),\ldots,\alpha(a_4))\). The measure
\[
 \mathfrak m(a)=(\omega(a),\alpha^*(a),\lambda(a))
\]
is ordered lexicographically. For a value \(\mathfrak m\) in degree \(d\),
\[
 W^d_{<\mathfrak m}
 =\operatorname{span}\{x^a:|a|=d,\ x^a\in\Hit^d(P(4)),
                              \ \mathfrak m(a)<\mathfrak m\}.
\]

\subsection{Duals, relative spaces, and group cohomology}

The divided-power dual is
\(\DP_d(n)=P^d(n)^*\), with basis \(v^{(a)}\) satisfying
\(\langle v^{(a)},x^b\rangle=\delta_{a,b}\).
The symmetric dual is \(\DS_d(n)=B^d(n)^*\), with basis
\(b_\lambda\) dual to \(m_\lambda\).
In the stable dual symmetric algebra, the notation
\(b_\lambda=b_{\lambda_1}\cdots b_{\lambda_s}\), with all parts
positive, is the ordinary polynomial product; \(\DS(n)\) is the
subspace of length at most \(n\).
The down squares \(\Sq_*^r\) are adjoint to \(\Sq^r\). In particular,
\begin{equation}\label{eq:dual-cartan}
 \Sq_*^r(b_j)=\binom{j-r}{r}b_{j-r},
 \qquad
 \Sq_*^r(uv)=\sum_{i=0}^r\Sq_*^i(u)\Sq_*^{r-i}(v),
\end{equation}
with a binomial coefficient zero if its indices are outside their
usual range. These operations preserve length
\cite[Section~26.1, equations~(26.2)--(26.3)]{WalkerWoodII}.
The ordinary and symmetric Steenrod kernels are
\[
\begin{aligned}
 K_d(n)&=\bigcap_{r>0}\ker(\Sq_*^r:\DP_d(n)\to\DP_{d-r}(n)),\\
 K_d^\Sigma(n)&=\bigcap_{r>0}\ker(\Sq_*^r:\DS_d(n)\to\DS_{d-r}(n)).
\end{aligned}
\]
An element of either kernel is called Steenrod-annihilated, meaning
that every positive down square kills it.
Inclusion induces \(\bar\iota_d:Q_d(B(n))\to Q_d(P(n))\),
and restriction gives \(\rho_d:K_d(n)\to K_d^\Sigma(n)\).
Define
\begin{equation}\label{eq:relative}
 \Crel_d(n)=\ker\bar\iota_d
  =\frac{B^d(n)\cap\Hit^d(P(n))}{\Hit^d(B(n))},
 \qquad
 \Rel_d(n)=K_d^\Sigma(n)/\rho_d(K_d(n)).
\end{equation}

The equivariant preimage spaces are
\[
 E_d(n)=\bigoplus_{1\leq r\leq\lfloor d/2\rfloor}P^{d-r}(n),
 \quad T_d((q_r)_r)=\sum_r\Sq^r q_r,\quad
 L_d(n)=\ker T_d,\quad I_d(n)=\Hit^d(P(n)).
\]
Empty sums are zero. For \(H\leq G_n\), put
\[
 \Crel_d(H)=I_d(n)^H/T_d(E_d(n)^H),\qquad
 M_H=\F[H\backslash G_n].
\]
Thus \(\Crel_d(H_n)=\Crel_d(n)\).
For a right \(\F H\)-module \(M\), our group-cohomology convention is
\[
\begin{aligned}
 Z^1(H,M)&=\{z:H\to M:z(hk)=z(h)k+z(k)\},\\
 H^1(H,M)&=Z^1(H,M)/\{h\mapsto mh-m:m\in M\}.
\end{aligned}
\]
The notation \(\Ext^1_{\F G_n}\) refers to right modules.
For \(n=4\), \(J\) is the all-ones \(4\)-by-\(4\) matrix and
\(\tau=I_4+J\).

For the comparison with ordinary representation theory, write
\(\operatorname{Mat}_n(\F)\) for the multiplicative monoid of all
\(n\)-by-\(n\) matrices over \(\F\). The transpose dual of a
finite-dimensional right matrix module \(V\) is \(V^*\) with action
\[
 \langle\varphi\cdot A,v\rangle
   =\langle\varphi,v\cdot A^{\mathsf T}\rangle;
\]
it is denoted by \(V^{\tr}\). This is the matrix-module convention
on \(\DP_d(n)=P^d(n)^*\)
\cite[Definition~9.2.2 and Proposition~9.2.3]{WalkerWoodI}.
Let \(J_d(n)\) be the \(\F G_n\)-submodule of \(K_d(n)\)
generated by the divided-power spikes \(v^{(a)}\), where \(|a|=d\)
and each \(a_i=2^{r_i}-1\) for some \(r_i\geq0\)
\cite[Definition~9.4.5]{WalkerWoodI}.
The strongly spike-free submodule \(\SSF_d(n)\) consists of
polynomials for which every matrix substitution has zero coefficient
on every spike monomial, and
\(\SF_d(n)=\SSF_d(n)/I_d(n)\)
\cite[Definition~30.2.1]{WalkerWoodII}.
The superscript \(*\) denotes linear duality; \(\tr\) specifies
the additional transpose-dual module convention.

\subsection{Support, Kameko maps, and the explicit families}

For \(S\subseteq\{1,\ldots,n\}\), \(P^+(S)\) is the span of monomials
with positive exponents precisely in \(S\).
Write \(P^+(r)=x_1\cdots x_rP(r)\),
\(B^+(r)=P^+(r)^{\Sigma_r}\), and
\[
 \Crel_d^+(r)=
 \frac{B^{+,d}(r)\cap\Hit^d(P^+(r))}{\Hit^d(B^+(r))}.
\]
For \(r=0\), both polynomial spaces are \(\F\) in degree zero.
The operator \(\mathcal T\) on three variables is the cyclic trace
\(\mathcal T(f)=f(x,y,z)+f(y,z,x)+f(z,x,y)\).

The up and down Kameko maps are
\[
 \upsilon_n(f)=x_1\cdots x_nf^2,\qquad
 \kappa_n(x^a)=
 \begin{cases}
  x_1^{(a_1-1)/2}\cdots x_n^{(a_n-1)/2},&\text{all }a_i\text{ odd},\\
  0,&\text{otherwise}.
 \end{cases}
\]
They are \(\F\)-linear. For \(t\geq0\), set
\[
 a^{[t]}=(2^t(a_1+1)-1,\ldots,2^t(a_n+1)-1),\qquad
 \Delta_t(e)=(e+4)2^t-4.
\]
The transpose of \(\kappa_n\) on the symmetric dual is denoted
\(\kappa_{n,B}^*\); this notation distinguishes it from the identity
Steenrod operation \(\Sq^0\).
For the fixed-preimage constructions, put
\[
 \gamma(q;a,b,c)=(q+a,q+a,b,c),\qquad
 \lambda(q;a,b,c)=(2q+a,2q+a,b,c).
\]
For \(a\geq3\) with \(a\equiv3\pmod4\), put
\[
 \theta(a)=
 \begin{cases}
 (a+6,a,1,1),&a\equiv3\pmod8,\\
 (a+4,a+2,1,1),&a\equiv7\pmod8.
 \end{cases}
\]
The particular tuples
\(\lambda_s=(2^s+1,2^s+1,3,1)\), \(s\geq3\), include
\(\lambda_3=(9,9,3,1)\).
The two-pair tuples and their iterates are
\[
\begin{aligned}
 A&=(21,21,14,14),& B&=(22,22,13,13),\\
 A_t=A^{[t]}&=(22\cdot2^t-1,22\cdot2^t-1,
                           15\cdot2^t-1,15\cdot2^t-1),\\
 B_t=B^{[t]}&=(23\cdot2^t-1,23\cdot2^t-1,
                           14\cdot2^t-1,14\cdot2^t-1).
\end{aligned}
\]

For the finite certificates, the three-variable polynomials are
\[
\begin{aligned}
 f_1&=\mathcal T(x^3y^4z^7+x^3y^6z^5),\\
 f_2&=\mathcal T(x^2y^3z^8+x^2y^5z^6+x^3y^3z^7+x^3y^5z^5),\\
 f_4&=\mathcal T(x^3y^3z^5).
\end{aligned}
\]
The corresponding degree-fifteen detector is
\(\Theta_{15}=b_5^3+b_6^2b_3+b_7b_5b_3+b_9b_3^2\).
For degree eighteen, the polynomials and functionals are
\[
 u=m_{7,7,4,0},\qquad
 v=m_{5,5,5,3}+m_{7,4,4,3}+m_{7,7,2,2},
\]
\[
\begin{aligned}
 \Theta_u={}&b_7^2b_4+b_8b_7b_3+b_9b_6b_3+b_9b_7b_2
 +b_{10}b_5b_3+b_{10}b_7b_1\\
 &+b_{11}b_4b_3+b_{11}b_5b_2+b_{11}b_6b_1+b_{12}b_3^2,\\
 \Theta_v={}&b_7^2b_2^2+b_1^2(b_9b_7+b_{11}b_5+b_{13}b_3).
\end{aligned}
\]

For the degree reductions, put
\[
 D_t^{(n)}(e)=(e+n)2^t-n,\qquad
 s_n(e)=\max\{0,n-1-\alpha(e+n-1)\}
 \quad(n\geq2,\ e,t\geq0).
\]
Thus \(D_t^{(4)}(e)=\Delta_t(e)\). The two degree families used in
the applications are
\[
 d_t=5\cdot2^t-4=\Delta_t(1),\qquad
 e_t=7\cdot2^t-4=\Delta_t(3)\qquad(t\geq0).
\]
For \(n\geq2\), let \(\mathscr P_E^{(n)}\) be the nonincreasing
\(n\)-tuples of nonnegative integers of sum \(E\). For \(E,u\geq0\)
and \(D=(E+n)2^u-n\), set
\[
 \mathscr R_{E,u}^{(n)}=\{a^{[u]}:a\in\mathscr P_E^{(n)}\},\qquad
 \mathscr Z_{E,u}^{(n)}=\mathscr P_D^{(n)}\setminus\mathscr R_{E,u}^{(n)}.
\]
The iterated up map includes the identity when \(u=0\) and satisfies
\begin{equation}\label{eq:iterated-up-formula}
 \upsilon_n^u(f)=(x_1\cdots x_n)^{2^u-1}f^{2^u}.
\end{equation}
The finite ordered lists of partitions needed for the principal
calculations are
\begin{equation}\label{eq:finite-partition-lists}
\begin{aligned}
 \mathscr A_1&=((1,0,0,0)),\\
 \mathscr A_3&=((1,1,1,0),(3,0,0,0)),\\
 \mathscr A_6&=((3,1,1,1),(3,3,0,0)),\\
 \mathscr A_{10}&=((3,3,2,2),(3,3,3,1),(7,1,1,1),(7,3,0,0)),\\
 \mathscr A_{16}&=((5,5,5,1),(7,3,3,3),(7,4,4,1),\\
 &\hspace{20mm}(7,7,1,1),(7,7,2,0),(15,1,0,0)).
\end{aligned}
\end{equation}
For \(a\) in one of these lists, the notation \([m_a]\) always
denotes its class in the symmetric cohit space of the corresponding
degree. The degree-sixteen symmetric dual functionals are
\begin{equation}\label{eq:sixteen-detectors}
\begin{aligned}
 \Psi_1&=b_1\Theta_{15},&
 \Psi_2&=b_7b_3^3,\\
 \Psi_3&=b_7b_4^2b_1+b_7b_5b_3b_1
                +b_{11}b_2^2b_1+b_{13}b_1^3,\\
 \Psi_4&=b_7^2b_1^2,&
 \Psi_6&=b_{15}b_1,\\
 \Psi_5&=b_7^2b_2+\sum_{j=8}^{14}b_jb_{15-j}b_1.
\end{aligned}
\end{equation}

Transposing \eqref{eq:cartan} gives the ordinary dual formula;
compare \cite[Definition~2.2, Proposition~2.3 and Theorem~2.1]{AziziJanfada2022}:
\begin{equation}\label{eq:ordinary-dual-cartan}
 \Sq_*^r v^{(a)}
 =\sum_{\substack{u\in\mathbb N^n,\ |u|=r\\2u_i\leq a_i\ (1\leq i\leq n)}}
    \left(\prod_{i=1}^n\binom{a_i-u_i}{u_i}\right)v^{(a-u)}.
\end{equation}
For the degree-sixteen certificates, the finite sets of ordered triples are
\[
\begin{aligned}
\mathcal U={}&\{(3,3,9),(3,6,6),(3,9,3),(5,3,7),(5,5,5),\\
 &\hspace{12mm}(5,7,3),(6,3,6),(6,6,3),(7,5,3),(9,3,3)\},\\
\mathcal V={}&\{(3,5,6),(3,6,5),(5,3,6),(5,6,3),(6,3,5),(6,5,3)\}.
\end{aligned}
\]
Under the tensor identification
\(\DP(4)=\DP(1)\otimes\DP(3)\), with
\(v^{(a)}\otimes v^{(b,c,d)}=v^{(a,b,c,d)}\), put
\begin{equation}\label{eq:degree-sixteen-ordinary-detectors}
\begin{aligned}
 U&=\sum_{a\in\mathcal U}v^{(a)}\in\DP_{15}(3),
 &V&=\sum_{a\in\mathcal V}v^{(a)}\in\DP_{14}(3),\\
 P_{16}&=v^{(1)}\otimes U+v^{(2)}\otimes V,
 &F_{16}&=\sum_{\substack{a,b,c\geq0\\a+b+c=15}}v^{(a,b,c,1)},\\
 \Pi_1&=P_{16},&\Pi_2&=v^{(3,3,3,7)},\\
 \Pi_3&=F_{16},&\Pi_4&=v^{(1,1,7,7)},\qquad
 \Pi_5=v^{(0,0,1,15)}.
\end{aligned}
\end{equation}
The substitution \(\pi_{16}:P(4)\to\F[y_1,y_2]\) is determined by
\[
 \pi_{16}(x_1)=\pi_{16}(x_2)=\pi_{16}(x_3)=y_1,
 \qquad\pi_{16}(x_4)=y_2.
\]
The ordinary source polynomials in three variables are
\begin{equation}\label{eq:degree-sixteen-ordinary-sources}
 h_r^{(3)}(x,y,z)=\sum_{(a,b,c)\in\mathcal D_r}x^ay^bz^c,
 \qquad r\in\{1,2,4\},
\end{equation}
where
\[
\begin{aligned}
\mathcal D_1={}&\{(1,4,10),(1,5,9),(1,7,7),(1,8,6),(3,2,10),\\
 &\quad(3,3,9),(3,10,2),(5,1,9),(5,4,6),(5,6,4),\\
 &\quad(5,9,1),(7,1,7),(7,2,6),(7,6,2),(7,7,1)\},\\
\mathcal D_2={}&\{(1,6,7),(1,7,6),(2,3,9),(3,2,9),(3,3,8),\\
 &\quad(3,5,6),(3,9,2),(6,1,7),(6,7,1),(7,1,6),(7,6,1)\},\\
\mathcal D_4={}&\{(1,5,6),(3,3,6),(3,5,4),(5,1,6),(5,6,1)\}.
\end{aligned}
\]
For \(1\leq j\leq4\), the map \(\eta_j:P(3)\to P(4)\)
substitutes the three variables, in increasing order, from
\(\{x_1,x_2,x_3,x_4\}\setminus\{x_j\}\).

\subsection{Finite matrices}

Matrix descriptions of Steenrod squares are developed in
\cite[Section~5]{Janfada2007}. Here the finite hit maps use
ordered monomial bases and monomial symmetric bases.
Let \(\mathscr E_d\) be the ordered four-tuples of nonnegative integers
of sum \(d\), and \(\mathscr P_d\) their nonincreasing representatives.
The dimensions are \(N_d=|\mathscr E_d|=\binom{d+3}{3}\) and
\(b_d=|\mathscr P_d|\).
The ordinary matrix \(O_d\) has columns
\(\Sq^r(x^a)\) in the \(\mathscr E_d\)-basis, with
\(a\in\mathscr E_{d-r}\) and
\(r\in\{1,2,4,\ldots\}\), \(r\leq d/2\).
The symmetric matrix \(S_d\) has columns
\(\Sq^r(m_\lambda)\) in the \(\mathscr P_d\)-basis with the same square
indices. The matrix \(A_d\) has columns \(m_\lambda\) in the ordered
monomial basis. The matrices are allowed to have zero columns.
Write
\[
 r_d=\rank O_d,\quad s_d=\rank S_d,\quad
 a_d=\rank[O_d\mid A_d],\quad c_d=\dim\Crel_d(4).
\]
A superscript \(\mathrm{all}\) means that every positive square
index \(r\leq d/2\) is used. The coefficient functional
\(\epsilon_\lambda\) is \(b_\lambda\) in the symmetric dual.

The entries of these matrices are computed from binomial coefficients
modulo two. We use the following binary form of Lucas's theorem,
recorded in \cite[Proposition~1.4.11]{WalkerWoodI}, to determine
which Cartan terms have nonzero coefficients.

\begin{lemma}\label{lem:lucas}
Let \(a,u\) be nonnegative integers. Then
\(\binom au=1\) in \(\F\) if and only if \(u\subseteq_2 a\).
\end{lemma}

We also recall naturality under linear substitutions
\cite[Proposition~1.2.5 and the following paragraph]{WalkerWoodI},
the spike coefficient criterion
\cite[Proposition~1.5.3]{WalkerWoodI}, and the non-hit property
in degrees one less than a power of two
\cite[Proposition~1.5.1]{WalkerWoodI}.

\begin{lemma}\label{lem:naturality-spikes}
Every substitution of variables by degree-one linear forms over
\(\F\) commutes with Steenrod squares. Every coefficient functional
supported on a spike belongs to the ordinary Steenrod kernel.
For \(n\geq1\) and \(j\geq0\), no monomial of degree
\(2^j-1\) in \(P(n)\) is hit.
\end{lemma}

The Peterson theorem and the minimal-spike criterion have the
following forms for permutation invariants; see
\cite[Theorem~25.1.6 and Proposition~25.1.5]{WalkerWoodII}.
The last assertion is the distinct-exponent case of Janfada's
monomial conjecture \cite[Theorem~3]{Janfada2011}, also recorded in
\cite[Proposition~25.1.2]{WalkerWoodII}.

\begin{proposition}\label{prop:classical}
Let \(n\geq1\), \(d\geq0\), and \(H\leq\Sigma_n\).
Then \(Q_d(P(n)^H)=0\) if and only if \(\mu(d)>n\).
If \(\mu(d)\leq n\) and
\(\omega(a)<_l\omega_{\min}(d)\), then
\(N_{H\cap H_a}^{H}(x^a)\) is hit in \(P(n)^H\).
If \(a\) has distinct entries and \(x^a\) is ordinary hit, then
\(m_a\) is symmetrically hit.
\end{proposition}

The next assertion is an elementary consequence of the top-square
identity, commutation with partial differentiation, and the non-hit
property in degree \(2^r-1\)
\cite[Propositions~1.1.10, 1.3.6 and~1.5.1]{WalkerWoodI}.
It will give the monomial test in degree sixteen directly.

\begin{proposition}\label{prop:power-of-two-monomials}
Let \(n\geq1\), \(r\geq0\), and \(a\in\mathbb N^n\)
satisfy \(|a|=2^r\). Then \(x^a\) is hit in \(P(n)\)
if and only if every exponent \(a_i\) is even. In this case
\(r\geq1\), and its symmetric orbit sum satisfies
\[
 m_a=\Sq^{2^{r-1}}(m_{a/2}).
\]
\end{proposition}
\begin{proof}
If all exponents are even, then \(r\geq1\), and the top-square
identity gives
\[
 x^a=\Sq^{2^{r-1}}(x^{a/2}),\qquad
 m_a=(m_{a/2})^2=\Sq^{2^{r-1}}(m_{a/2}).
\]
Here halving preserves the equalities between exponents, so it
identifies the two monomial orbits.
Conversely, if \(a_i\) is odd, then
\(\partial x^a/\partial x_i=x^{a-\mathbf e_i}\) is a nonzero
monomial of degree \(2^r-1\), where \(\mathbf e_i\) is the
\(i\)-th standard basis vector. Partial differentiation commutes
with every Steenrod square by
\cite[Proposition~1.3.6]{WalkerWoodI}, and therefore takes hits to
hits. The non-hit property in Lemma~\ref{lem:naturality-spikes}
excludes a hit expression for \(x^a\).
\end{proof}

\subsection{Kameko reduction for permutation invariants}

The numerical condition in Proposition~\ref{prop:classical} has a
second consequence: when the minimum number of spike summands is
exactly the number of variables, the degree can be reduced by a
Kameko isomorphism. For ordinary cohits this is the classical Kameko
reduction; see \cite[Propositions~6.5.2--6.5.3]{WalkerWoodI}
and \cite[Theorem~2]{TinSum2016}. Janfada and Wood proved
the symmetric counterpart \cite[Theorem~2.3]{JanfadaWood2002}, which
is restated in iterated form in
\cite[Theorem~2.4]{JanfadaCriteriaB3}. The formulation for every
permutation subgroup is \cite[Proposition~25.1.9]{WalkerWoodII}.
The invariant cohit space here is \(Q_d(P(n)^H)\); taking
\(H\)-invariants of \(Q_d(P(n))\) is a different operation.

\begin{proposition}\label{prop:kameko-cohit}
Let \(n\geq1\), \(d\geq0\), \(D=2d+n\), and \(H\leq\Sigma_n\).
Then \(\kappa_n\) induces a map
\(Q_D(P(n)^H)\to Q_d(P(n)^H)\). If \(\mu(D)=n\), then this map
is an isomorphism with inverse induced by \(\upsilon_n\).
The polynomial map \(\kappa_n\) is \(G_n\)-equivariant and satisfies,
for every \(r\geq0\),
\begin{equation}\label{eq:kameko-squares}
 \kappa_n\Sq^{2r}=\Sq^r\kappa_n,\qquad
 \kappa_n\Sq^{2r+1}=0.
\end{equation}
\end{proposition}

In four variables, Proposition~\ref{prop:kameko-cohit} gives
\begin{equation}\label{eq:symmetric-kameko-four}
 \mu(2d+4)=4\quad\Longrightarrow\quad
 Q_{2d+4}(B(4))\xrightarrow[\kappa_4]{\ \cong\ }Q_d(B(4)).
\end{equation}
Equivalently, if \(D\geq4\) and \(\mu(D)=4\), the lower degree is
\((D-4)/2\). Since every positive spike exponent is odd,
\(\mu(D)=4\) forces \(D\) even. This condition is sufficient for
the isomorphism. It need not hold at a degree that serves as the
initial degree for an infinite sequence of isomorphisms.

\section{Relative duality and equivariant preimage obstructions}
\label{sec:relative}

The definitions separate the ordinary hit space from the requirement
that its preimages be fixed. We now identify the obstruction to fixed
preimages by a group-cohomology class. The same construction explains
the available normalizer action and its relation to ordinary
\(\GL(n,\F)\)-modules.

\begin{theorem}\label{thm:relative-structure}
Let \(n\geq1\), \(d\geq0\), and \(H\leq G_n\).
Then there are natural isomorphisms
\begin{align}
 \Crel_d(n)^*&\cong\Rel_d(n),\label{eq:duality}\\
 \Crel_d(H)&\cong
 \ker\bigl(H^1(H,L_d(n))\longrightarrow H^1(H,E_d(n))\bigr),
 \label{eq:h1}\\
 \Crel_d(H)&\cong
 \ker\bigl(\Ext^1_{\F G_n}(M_H,L_d(n))
                  \longrightarrow\Ext^1_{\F G_n}(M_H,E_d(n))\bigr).
 \label{eq:ext}
\end{align}
The pairing in \eqref{eq:duality} is
\(([b],[\Theta])\mapsto\Theta(b)\).
The isomorphism \eqref{eq:h1} is equivariant for
\(N_{G_n}(H)/H\).
If \(b\in I_d(n)^H\) and \(T_d(q)=b\), then its class in
\eqref{eq:h1} is represented by \(h\mapsto qh-q\).
\end{theorem}

\begin{proof}
The standard kernel--cohit dualities are
\[
 Q_d(P(n))^*=K_d(n),\qquad
 Q_d(B(n))^*=K_d^\Sigma(n);
\]
see \cite[Proposition~9.4.2(iii)]{WalkerWoodI} and
\cite[Section~26.1]{WalkerWoodII}.
Under these identifications, the transpose of \(\bar\iota_d\) is
\(\rho_d\). The finite-dimensional identity
\(\operatorname{coker}T^*\cong(\ker T)^*\) therefore proves
\eqref{eq:duality}. The displayed pairing is independent of the
representative of \(b\), because \(\Theta\) annihilates symmetric
hits, and independent of the representative of \(\Theta\), because
ordinary kernel functionals annihilate \(b\).

Lemma~\ref{lem:naturality-spikes} makes every square and hence
\(T_d\) equivariant under \(G_n\). Its image is \(I_d(n)\) by the leftmost-square description
and instability. Therefore
\[
 0\longrightarrow L_d(n)\longrightarrow E_d(n)
 \xrightarrow{T_d} I_d(n)\longrightarrow0
\]
is exact as a sequence of \(G_n\)-modules.
Moreover, \(T_d(E_d(n)^H)\) is exactly the hit subspace with
\(H\)-fixed preimages.

Choose \(q\) with \(T_d(q)=b\).
For \(h\in H\), equivariance gives \(T_d(qh-q)=bh-b=0\).
Thus \(z_b(h)=qh-q\) lies in \(L_d(n)\), and
\[
 z_b(hk)=(qh-q)k+(qk-q)=z_b(h)k+z_b(k).
\]
Replacing \(q\) by \(q+\ell\), \(\ell\in L_d(n)\), changes \(z_b\)
by the coboundary \(h\mapsto\ell h-\ell\).
The class of \(z_b\) is therefore well-defined.
It is zero precisely when \(qh-q=\ell h-\ell\) for some \(\ell\)
and all \(h\), equivalently \(q-\ell\in E_d(n)^H\).
Its kernel as a function of \(b\) is consequently
\(T_d(E_d(n)^H)\).

The image of \(z_b\) in \(H^1(H,E_d(n))\) is zero because it is the
coboundary of \(q\). Conversely, if a cocycle
\(z:H\to L_d(n)\) becomes a coboundary in \(E_d(n)\), write
\(z(h)=qh-q\). Applying \(T_d\) proves \(T_d(q)h=T_d(q)\).
Hence \(b=T_d(q)\) belongs to \(I_d(n)^H\) and represents the
given class. This proves \eqref{eq:h1}.

For \(a\in N_{G_n}(H)\), the action on cocycles is
\((z\cdot a)(h)=z(aha^{-1})a\).
A preimage of \(ba\) is \(qa\), and
\[
 qa h-qa=\bigl(q(aha^{-1})-q\bigr)a.
\]
This proves equivariance. If \(a=k\in H\), the cocycle identity gives
\((z\cdot k)(h)-z(h)=z(k)h-z(k)\), a coboundary.
The action therefore factors through the normalizer quotient.

Finally, the induction adjunction in group cohomology
\cite[Chapter~III]{Brown1982} gives natural isomorphisms
\[
 \Ext^1_{\F G_n}(\F\otimes_{\F H}\F G_n,M)
 \cong H^1(H,M).
\]
Since \(M_H=\F\otimes_{\F H}\F G_n\), naturality with respect
to \(L_d(n)\hookrightarrow E_d(n)\) identifies the two kernels
in \eqref{eq:h1} and \eqref{eq:ext}.

\end{proof}

\begin{corollary}\label{cor:relative-detection}
Let \(b\in B^d(n)\cap\Hit^d(P(n))\).
Then \(b\) is symmetrically hit if and only if
\(\Theta(b)=0\) for every \([\Theta]\in\Rel_d(n)\).
Every \(\Psi\in K_d(n)\) satisfies \(\Psi(b)=0\).
In particular, if \(\Crel_d(n)=0\), then every ordinary-hit
monomial of degree \(d\) has a symmetrically hit orbit sum.
\end{corollary}
\begin{proof}
The first assertion is nondegeneracy of the pairing in
Theorem~\ref{thm:relative-structure}. The second is the defining
annihilator property of \(K_d(n)\). For the last assertion,
each permutation translate of an ordinary hit monomial is hit,
because permutations commute with squares. Their distinct orbit
sum is therefore an ordinary hit. If the relative space is zero,
its symmetric cohit class must also be zero.
\end{proof}

The usual transfer argument for a subgroup of odd index
\cite[Chapter~III]{Brown1982} has the following consequence for
relative cohits. Its odd-order case is
\cite[Proposition~25.1.1]{WalkerWoodII}.

\begin{proposition}\label{prop:odd-index-preimages}
Let \(n\geq1\), \(d\geq0\), and \(K\leq H\leq G_n\), with
\([H:K]\) odd. Inclusion of fixed subspaces induces an injective map
\[
 \Crel_d(H)\longrightarrow\Crel_d(K).
\]
Consequently, an element of \(I_d(n)^H\) has an \(H\)-fixed
Steenrod preimage if and only if it has a \(K\)-fixed Steenrod
preimage. In particular, the condition may be checked on a Sylow
\(2\)-subgroup of \(H\), and \(\Crel_d(H)=0\) if \(H\) has odd order.
\end{proposition}
\begin{proof}
Inclusion of fixed subspaces induces the displayed map. If its
kernel contains the class of \(b\in I_d(n)^H\), choose
\(q\in E_d(n)^K\) with \(T_d(q)=b\). Applying the coset norm
componentwise gives an \(H\)-fixed preimage, since
\[
 T_d\bigl(N_K^H(q)\bigr)
 =N_K^H\bigl(T_d(q)\bigr)=[H:K]b=b.
\]
This proves injectivity. An \(H\)-fixed preimage is already
\(K\)-fixed, proving the equivalence. A Sylow \(2\)-subgroup
has odd index. For a group of odd order, take \(K=\{1\}\),
whose relative space is zero.
\end{proof}

\begin{proposition}\label{prop:normalizer}
Let \(H_4\) be the permutation subgroup of \(G_4\).
Then
\[
 N_{G_4}(H_4)=H_4\times\langle\tau\rangle,
 \qquad\langle\tau\rangle\cong C_2.
\]
For \(v=\dim\Crel_d(4)\) and
\(b=\rank(\tau-1\mid\Crel_d(4))\), there is an isomorphism
\[
 \Crel_d(4)\cong\F^{\,v-2b}\oplus(\F C_2)^b
\]
of modules for the normalizer quotient.
\end{proposition}
\begin{proof}
Every automorphism of \(\Sigma_4\) preserves transpositions, since
their centralizers have order four, whereas double transpositions
have centralizers of order eight. Three pairwise noncommuting
transpositions form a triangle or a star in the complete graph on
four letters. A triangle generates a subgroup of order six;
a star generates \(\Sigma_4\). Automorphisms therefore permute the
four stars. The intersection of the stars centered at \(i\) and
\(j\) is the single transposition \((ij)\). The induced permutation
of the centers determines the automorphism on every transposition,
so every automorphism is inner.

It follows that the normalizer is the product of \(H_4\) with its
centralizer in \(G_4\). A matrix commuting with every permutation
matrix has constant diagonal entries and constant off-diagonal
entries, and therefore has form \(\alpha I_4+\beta J\).
For \(\alpha=0\) it is singular. For \(\alpha=1\) it is either
\(I_4\) or \(\tau\), both invertible since \(J^2=0\).
The matrix \(\tau\) is not a permutation matrix. This proves the
direct product assertion.

The substitution \(x_i\mapsto\sum_{j\ne i}x_j\) associated to
\(\tau\) preserves symmetric polynomials and both hit spaces,
so it acts on \(\Crel_d(4)\).
On this space \(N=\tau-1\) satisfies \(N^2=0\).
Choose a basis \(w_1,\ldots,w_b\) of \(\im N\), choose \(u_i\)
with \(Nu_i=w_i\), and extend the \(w_i\) to a basis of
\(\ker N\) with vectors \(z_1,\ldots,z_{v-2b}\).
Applying \(N\) to a linear relation among \(u_i,w_i,z_j\)
first makes all coefficients of the \(u_i\) zero, after which
the basis of \(\ker N\) makes the remaining coefficients zero.
These vectors thus form a basis. Each span of \(u_i,w_i\)
is a regular two-dimensional \(\F C_2\)-module, and each \(z_j\)
spans a trivial module.
\end{proof}

\textit{Remark.} The whole group \(G_4\) does not preserve \(B(4)\).
For example, the transvection \(x_1\mapsto x_1+x_2\), fixing the
other variables, sends \(\e_1\) to \(x_1+x_3+x_4\).
Thus a decomposition of the relative space into irreducible
\(G_4\)-modules is not supplied by its definition.
The extension-group description \eqref{eq:ext} gives the
precise connection with \(G_4\)-representation theory.

Walker and Wood prove that \(\SF_d(n)\) is an ordinary cohit
submodule and identify it with
\((K_d(n)/J_d(n))^{\tr}\)
\cite[Proposition~30.2.2]{WalkerWoodII}; the corresponding local
statement for decreasing weights is
\cite[Proposition~30.2.4]{WalkerWoodII}.
A relative class has zero image in the entire ordinary cohit
module, so it cannot be identified through inclusion with a
nonzero strongly spike-free ordinary class.
There is also a precise projectivity consequence of
\eqref{eq:ext}. If \(M_{H_4}=M_0\oplus M_1\) with \(M_1\)
projective, then \(\Ext^1(M_1,-)=0\), and \(M_1\) makes no
contribution to that kernel. In particular this applies to any
Steinberg direct summand, since the Steinberg module is projective
\cite[Proposition~18.4.9(ii)]{WalkerWoodII}.
No multiplicity of such a summand is needed here.

\begin{proposition}\label{prop:support}
Let \(n\geq1\) and \(d\geq0\). Then
\[
 \Crel_d(n)\cong\bigoplus_{r=0}^n\Crel_d^+(r).
\]
For \(m\leq n\), extension of monomial symmetric functions from
\(m\) to \(n\) variables induces a split injection
\(\Crel_d(m)\hookrightarrow\Crel_d(n)\).
\end{proposition}
\begin{proof}
Steenrod squares preserve the support of monomials
\cite[Proposition~1.1.11]{WalkerWoodI}. The support decomposition
and the symmetric length decomposition
\cite[Proposition~26.1.2(iv)]{WalkerWoodII} give
\[
 P(n)=\bigoplus_SP^+(S),\qquad
 \Hit(P(n))=\bigoplus_S\Hit(P^+(S)),\qquad
 B(n)\cong\bigoplus_{r=0}^nB^+(r).
\]
In the last isomorphism, a component on \(\{1,\ldots,r\}\)
extends by summing its relabellings over all supports of size \(r\).
The ordinary direct sum shows that the extension is ordinary hit
exactly when its component on one support is ordinary hit.
The symmetric direct sum gives the corresponding decomposition
of symmetric hits. Taking their quotients proves the asserted
relative decomposition. Extension from \(m\) to \(n\) variables
includes the summands with \(r\leq m\); setting the extra variables
to zero projects onto them. Their composite is the identity.
\end{proof}

\section{Stabilizer parity and unconditional preimages}
\label{sec:stabilizers}

The cohomological obstruction concerns the existence of fixed preimages.
We next construct such preimages directly. Orbit counting determines
which Cartan terms survive symmetrization, while separation of binary
digits forces a prescribed term in an ordinary square.
The final criterion records the descent still required outside
the families treated here. The fixed-preimage symmetrization argument
is due to Janfada \cite[Proposition~2]{Janfada2011}; see also
\cite[Proposition~3.2]{Janfada2009Sym}. We make its orbit
coefficients explicit below. Proposition~\ref{prop:odd-index-preimages}
shows that fixedness can be tested on a Sylow \(2\)-subgroup
of the monomial stabilizer.

\begin{lemma}\label{lem:parity}
Let \(G=\Sigma_4\), let \(u,v\) be monomials of the same degree,
and let \(H=\Stab_G(u)\), \(K=\Stab_G(v)\), and \(L=H\cap K\).
If \(q_r\) and \(R\) are \(H\)-fixed polynomials satisfying
\[
 u+\sum_{h\in L\backslash H}v\cdot h+R
   =\sum_{r>0}\Sq^r q_r,
\]
then
\[
 \sigma(u)\equiv[K:L]\sigma(v)+N_H^G(R)
                       \pmod{\Hit(B(4))}.
\]
The contribution of the displayed \(H\)-orbit of \(v\) to
its \(G\)-orbit sum is nonzero in \(\F\) if and only if
\([K:L]\) is odd; contributions from \(R\) are separate.
\end{lemma}
\begin{proof}
By the fixed-preimage symmetrization of Janfada
\cite[Proposition~3.2]{Janfada2009Sym}, applying \(N_H^G\)
to the hypothesized equation gives a symmetric hit equation.
The displayed \(H\)-orbit of \(v\) contributes
\[
 N_H^G\left(\sum_{h\in L\backslash H}v\cdot h\right)
 =\sum_{g\in L\backslash G}v\cdot g
 =[K:L]\sigma(v).
\]
Transitivity of coset sums gives the first equality. Each
\(K\)-coset contains exactly \([K:L]\) \(L\)-cosets, and
\(v\) is \(K\)-fixed, giving the second. The leading monomial
contributes \(\sigma(u)\), and the remaining term contributes
\(N_H^G(R)\). Reduction modulo symmetric hits proves the formula;
reduction of \([K:L]\) modulo two gives the parity assertion.
\end{proof}

\textit{Remark.} For exponent multiplicities \(1+1+1+1,2+1+1,2+2,3+1,4\),
the stabilizer orders are respectively \(1,2,4,6,24\), and the
orbit sizes are \(24,12,6,4,1\).
The intersection in Lemma~\ref{lem:parity} is essential.
For example, with \(H=\langle(12),(34)\rangle\), a target
stabilizer \(\langle(34)\rangle\) has index one over the
intersection, whereas \(\langle(13)\rangle\) has index two.
The former contribution survives and the latter cancels.
The fixedness assumption on the preimages is indispensable:
the norm is not well-defined on arbitrary polynomials.

\begin{proposition}\label{prop:one-square}
Let \(q\) be a positive power of two and let \(a,b,c\) be
nonnegative integers with \(2a+b+c<q\). Then
\begin{align}
 \Sq^{2q}x^{\gamma(q;a,b,c)}
   &=x^{\lambda(q;a,b,c)},\label{eq:one-ordinary}\\
 \Sq^{2q}m_{\gamma(q;a,b,c)}
   &=m_{\lambda(q;a,b,c)}.\label{eq:one-symmetric}
\end{align}
The source and target stabilizers are equal.
If \(b\ne c\), this stabilizer is generated by \((12)\).
\end{proposition}
\begin{proof}
The inequality implies \(a,b,c<q\).
In each of the first two source exponents, the binary digit \(q\)
is one, and all lower digits are those of \(a\). The other two
exponents are smaller than \(q\).
By Lemma~\ref{lem:lucas}, every admissible increment has a
contribution \(q\) or zero in each of the first two coordinates,
together with lower increments whose total is at most
\(2a+b+c<q\).
If fewer than two high digits are selected, the total increment
is strictly smaller than \(2q\).
If both are selected, attaining total \(2q\) requires every lower
increment to be zero.
The sole surviving increment tuple is therefore \((q,q,0,0)\),
with coefficient one. This proves \eqref{eq:one-ordinary}.

Both high source exponents are at least \(q\), and the other
exponents are smaller than \(q\). The analogous strict separation
holds for the target, whose two high exponents are at least \(2q\).
Consequently equality of two source exponents is equivalent to
equality of the corresponding target exponents.
The stabilizers are equal. Summing \eqref{eq:one-ordinary}
over their common cosets gives \eqref{eq:one-symmetric},
with multiplicity one as in Lemma~\ref{lem:parity}.
If \(b\ne c\), only the first two exponents coincide.
\end{proof}

\begin{theorem}\label{thm:unconditional-pair}
Let \(r\geq3\), \(k\geq r+1\), and \(t\geq0\).
Let \(a,b,c\) be positive odd integers such that
\(b\ne c\) and \(2a+b+c=2^r-2\), and let \(q=2^k\).
Then
\begin{align}
 \Sq^{\,2q2^t}x^{\gamma(q;a,b,c)^{[t]}}
     &=x^{\lambda(q;a,b,c)^{[t]}},\label{eq:lifted-one-ordinary}\\
 \Sq^{\,2q2^t}m_{\gamma(q;a,b,c)^{[t]}}
     &=m_{\lambda(q;a,b,c)^{[t]}}.\label{eq:lifted-one-symmetric}
\end{align}
The source and target have stabilizer \(\Sigma_2\).
For \(d=|\lambda(q;a,b,c)^{[t]}|\),
\[
 \omega\bigl(\lambda(q;a,b,c)^{[t]}\bigr)>_l\omega_{\min}(d).
\]
\end{theorem}
\begin{proof}
Let \(Q=q2^t\), \(A=(a+1)2^t-1\),
\(B=(b+1)2^t-1\), and \(C=(c+1)2^t-1\).
The shifted source has exponents \((Q+A,Q+A,B,C)\).
Its lower increment bound is
\[
 2A+B+C=(2^r+2)2^t-4<q2^t=Q,
\]
because \(q\geq2^{r+1}>2^r+2\).
Proposition~\ref{prop:one-square} with \(Q,A,B,C\) proves
both identities. Strict monotonicity of
\(s\mapsto2^t(s+1)-1\) preserves all equality patterns,
so the stabilizer is \(\Sigma_2\).

For \(t=0\), the degree is \(2^{k+2}+2^r-2\), with
minimal spike \((2^{k+2}-1,2^r-1,0,0)\).
It is positive and even, and \eqref{eq:mu} gives \(\mu(d)=2\).
The target has four odd exponents, whereas this minimal spike
has two. Its weight is therefore strictly larger.
For \(t\geq1\),
\[
 d=(2^{k+2}+2^r+2)2^t-4.
\]
The quantity \(d+4\) has three nonzero binary digits, and
\[
 d+2=(2^{k+2}+2^r+1)2^t+(2^t-2)
\]
has \(t+2>2\) nonzero binary digits.
Equation~\eqref{eq:mu}, together with the even parity of \(d\),
therefore gives \(\mu(d)=4\).
The standard minimal spike has exponents
\[
 (2^{k+t+2}-1,\ 2^{r+t}-1,\ 2^t-1,\ 2^t-1).
\]
The associated positive lengths are strictly decreasing except
for the final equal pair, as required for a minimal spike
\cite[Section~5.4]{WalkerWoodI}.
Its first \(t\) weight entries are four and its next is two.
The target has first \(t+1\) weight entries equal to four,
because the unshifted exponents are all odd.
This proves the strict weight comparison.
\end{proof}

\begin{proposition}\label{prop:two-square}
Let \(a\geq3\) satisfy \(a\equiv3\pmod4\).
Then \(x^{\theta(a)}\) is ordinary hit and
\begin{equation}\label{eq:two-square}
 m_{\theta(a)}
   =\Sq^4m_{a+2,a,1,1}+\Sq^1m_{a+2,a+1,2,2}.
\end{equation}
The target stabilizer is \(\Sigma_2\).
If \(\alpha(a+5)\leq2\), then
\(\omega(\theta(a))>_l\omega_{\min}(2a+8)\).
\end{proposition}
\begin{proof}
Let \(b=a+2\). The last two digits of \(a\) are \(11\),
whereas those of \(b\) are \(01\).
Their digits in the \(4\)-position are complementary.
The admissible increment tuples of total four for
\((b,a,1,1)\) consist of exactly one of
\((4,0,0,0),(0,4,0,0)\), and
\[
 (1,3,0,0),\ (0,3,1,0),\ (0,3,0,1),\
 (1,2,1,0),\ (1,2,0,1),\ (0,2,1,1),\ (1,1,1,1).
\]
This is exhaustive because only the second coordinate admits
increment two, the last two increments are at most one, and
increments larger than four cannot occur.
For the twelve source monomials, the integer counts are
\[
\begin{array}{c|r|r|r}
\text{target}&\text{total terms}&\text{orbit size}&\text{multiplicity}\\
\hline
\theta(a)&12&12&1\\
(a+3,a+3,1,1)&12&6&2\\
(b+1,b,2,1)&48&24&2\\
(b,b,2,2)&12&6&2\\
(b+1,a+1,2,2)&12&12&1
\end{array}
\]
The second row comes from \((1,3,0,0)\).
The third combines \((0,3,1,0)\) and \((1,2,1,0)\)
and their exchanges of the last two coordinates, using \(a+2=b\).
The fourth and fifth rows come from \((0,2,1,1)\) and
\((1,1,1,1)\), respectively.
The output is symmetric, so the multiplicity is constant on
each target orbit and is the total divided by the orbit size.
Reducing modulo two gives
\[
 \Sq^4m_{b,a,1,1}
   =m_{\theta(a)}+m_{b+1,a+1,2,2}.
\]
Only the first exponent in \((b,a+1,2,2)\) is odd.
Thus \(\Sq^1m_{b,a+1,2,2}=m_{b+1,a+1,2,2}\).
Adding these identities proves \eqref{eq:two-square}.

In both residue cases, the four target exponents modulo eight
are a permutation of \((1,3,1,1)\).
Write \(x^{\theta(a)}=gh^8\), with these residues as the
exponents of \(g\). Then \(\deg g=6\) and \(\deg h>0\).
The Silverman--Singer criterion states that \(gh^{2^j}\)
is hit if \(\deg g<(2^j-1)\mu(\deg h)\)
\cite[Theorem~14.1.2]{WalkerWoodI}.
Here \(6<7\mu(\deg h)\), because \(\mu(\deg h)\geq1\).
This proves the ordinary assertion independently of symmetrization.
The first two target exponents are distinct and exceed one,
so the stabilizer is precisely the permutation of the last pair.

If \(d=2a+8\) and \(\alpha(a+5)\leq2\), then
\(\alpha(d+2)=\alpha(a+5)\leq2\).
Since \(d\) is positive and even, \eqref{eq:mu} implies
\(\mu(d)=2\).
The four target exponents are odd, giving first weight entry
four, while the minimal spike has first entry two.
\end{proof}

\textit{Remark.} For example, Proposition~\ref{prop:two-square} gives
\[
 m_{9,3,1,1}=\Sq^4m_{5,3,1,1}+\Sq^1m_{5,4,2,2},
 \qquad
 m_{11,9,1,1}=\Sq^4m_{9,7,1,1}+\Sq^1m_{9,8,2,2}.
\]
Both are above the minimal-spike weight.

\begin{proposition}\label{prop:pairpair-seventy}
Let \(H=\langle(12),(34)\rangle\).
Then \(x^A,x^B\) are ordinary hit, \(m_A,m_B\) are
symmetrically hit, and
\[
 \omega(A)=\omega(B)=(2,2,4,2,2)>_l\omega_{\min}(70).
\]
Both stabilizers are \(\Sigma_2\times\Sigma_2\).
Both ordered monomials admit positive-square preimages fixed by \(H\).
\end{proposition}
\begin{proof}
For \(e\geq4\), let \(\mathscr H_e\) be the ordered positive
four-tuples of sum \(e\) satisfying \(a_1\geq a_2\) and
\(a_3\geq a_4\), and put
\(p_a^H=\sum_{b\in H\cdot a}x^b\), with distinct orbit terms.
These polynomials form a basis of the full-support part of
\(P^e(4)^H\). In the basis indexed by \(\mathscr H_{70}\),
form the columns \(\Sq^r p_a^H\), where
\(r=1,2,4,8,16\) and \(a\in\mathscr H_{70-r}\).
The complete construction and exact elimination are given in
Appendix~\ref{app:fixed-verifier}. The row count is \(13685\),
and the successive calculations give
\[
\begin{array}{c|r|r}
r&|\mathscr H_{70-r}|&\text{Cumulative rank}\\\hline
1&13090&6681\\
2&12529&11666\\
4&11440&13306\\
8&9455&13571\\
16&6201&13632
\end{array}
\]
Retaining the source labels during elimination expresses \(x^A\)
and \(x^B\) using respectively \(249\) and \(251\) of these
\(H\)-fixed source polynomials. The program prints the complete
source lists and independently expands every indicated square in
the ordered monomial basis; the surviving sets are exactly
\(\{A\}\) and \(\{B\}\). Thus both targets are ordinary hits
with \(H\)-fixed preimages. No symmetric projection is used in
this last coefficient check.
The following symmetric identity is independent of that
membership test:
\begin{align*}
m_A={}&
\Sq^1\bigl(
 m_{18,18,17,16}+m_{19,17,17,16}+m_{19,19,16,15}
 +m_{19,19,17,14}\\
&\hspace{13mm}
 +m_{19,19,19,12}+m_{21,18,16,14}+m_{21,20,14,14}
 +m_{21,20,16,12}+m_{22,22,13,12}\bigr)\\
&+\Sq^2\bigl(
 m_{18,17,17,16}+m_{19,18,16,15}+m_{19,18,17,14}
 +m_{19,19,16,14}\\
&\hspace{13mm}
 +m_{20,19,15,14}+m_{21,21,14,12}\bigr)\\
&+\Sq^4\bigl(m_{19,18,15,14}+m_{19,19,14,14}\bigr)
 +\Sq^8m_{19,19,12,12}.
\end{align*}
For a source partition \(\eta\) and square index \(r\), its
coefficient on an ordered target \(\nu\) is exactly
\begin{equation}\label{eq:coefficient-test}
 \sum_{\substack{\xi\text{ a distinct permutation of }\eta\\
                         \xi_i\leq\nu_i\ (1\leq i\leq4)}}
 \prod_{i=1}^4\binom{\xi_i}{\nu_i-\xi_i},
 \qquad |\nu|-|\eta|=r.
\end{equation}
Applying this formula to the eighteen displayed sources gives
coefficient one precisely on the six permutations of \(A\),
and zero on every other ordered target.
Thus the display is an exact symmetric hit identity.
The two further Cartan expansions give
\begin{equation}\label{eq:pairpair-difference}
 m_A+m_B=\Sq^1m_{21,21,14,13}+\Sq^2m_{21,21,13,13}.
\end{equation}
Indeed, the first square is \(m_{22,21,14,13}\), while the
second is \(m_B+m_{22,21,14,13}+m_A\).
These expansions follow because \(21\) and \(13\) have
second binary digit zero, and the only increments in
\(\Sq^2\) are two increments of one.
Equation~\eqref{eq:pairpair-difference} proves that \(m_B\)
is also symmetrically hit.
The common weight follows by writing \(21,14,22,13\) in binary.
The minimal spike in degree \(70\) is \((63,7,0,0)\), with
weight \((2,2,2,1,1,1)\).
Finally, both equality patterns consist of two distinct
pairs, which gives the stated stabilizers.
\end{proof}

\begin{proposition}\label{prop:descent}
Let \(d\) satisfy \(\mu(d)\leq4\).
Suppose that each ordinary-hit monomial \(x^a\) of degree \(d\)
with repeated exponents and
\(\omega(a)\not<_l\omega_{\min}(d)\) admits an equality
\[
 m_a=h_B+\Orb(p_<)+R,
 \quad h_B\in\Hit^d(B(4)),\quad
 p_<\in W^d_{<\mathfrak m(a)},\quad
 R\in B^d(4)\cap\Hit^d(P(4)),
\]
in which \(\Theta(R)=0\) for every \([\Theta]\in\Rel_d(4)\).
Then \(m_a\) is symmetrically hit for every ordinary-hit
monomial \(x^a\) in degree \(d\).
\end{proposition}
\begin{proof}
There are finitely many exponent partitions in degree \(d\),
so the order on \(\mathfrak m\) is a finite induction order.
If a monomial lies below the minimal-spike weight or has
distinct exponents, Proposition~\ref{prop:classical} applies.
In every other case use the displayed equality.
Every monomial of \(p_<\) is an ordinary hit with smaller
measure, so the induction hypothesis makes its orbit sum
symmetrically hit. By Corollary~\ref{cor:relative-detection},
the assumption on \(R\) makes \(R\) symmetrically hit.
All three terms on the right are therefore symmetric hits.
This completes the induction.
\end{proof}

\textit{Remark.} The sufficient condition in Proposition~\ref{prop:descent}
is satisfied unconditionally by the explicit identities above.
For a further specified tuple it is a deterministic finite
linear problem: form the symmetric Cartan matrix, append
the orbit sums of the lower ordinary-hit
monomials, and perform elimination while retaining the
source labels. A successful elimination returns the actual
polynomials \(h_B\) and \(p_<\). Theorem~\ref{thm:relative-structure}
identifies the obstruction to a fixed preimage if the corresponding
fixed-preimage system has no solution.

\section{Explicit relative classes and the normalizer action}
\label{sec:certificates}

The fixed-preimage results give vanishing for individual orbit classes.
We now show that the full relative space can nevertheless be nonzero.
A short cyclic identity and its dual detector establish the first
obstruction, and a degree-eighteen calculation exhibits a nontrivial
extension under the canonical normalizer action.

\begin{theorem}\label{thm:fifteen}
Let \(n\geq3\). Then the length-three monomial symmetric function
\(m_{7,5,3}\), interpreted in \(n\) variables, is ordinary hit
and is not symmetrically hit.
Every ordered monomial of degree \(15\) is ordinary non-hit.
In four variables, \(\Crel_{15}(4)\) is one-dimensional and is
generated by \([m_{7,5,3,0}]\).
\end{theorem}
\begin{proof}
In three variables, direct Cartan expansion gives
\begin{equation}\label{eq:fifteen-cartan}
 m_{7,5,3}=\Sq^1f_1+\Sq^2f_2+\Sq^4f_4.
\end{equation}
Here is a complete verification after collecting cyclic
orbits. In the table, a triple \((i,j,k)\) denotes
\(\mathcal T(x^iy^jz^k)\); each column lists its coefficient.
\[
\begin{array}{c|ccc}
(i,j,k)&\Sq^1 f_1&\Sq^2 f_2&\Sq^4 f_4\\
\hline
(3,3,9)&0&1&1\\
(3,4,8)&1&1&0\\
(3,5,7)&0&1&0\\
(3,6,6)&1&1&0\\
(3,7,5)&0&1&0\\
(4,4,7)&1&1&0\\
(4,6,5)&1&1&0\\
(5,5,5)&0&1&1
\end{array}
\]
These are all the nonzero cyclic orbits produced by
\eqref{eq:cartan}. For example, in \(\Sq^1 f_1\) the
even exponent four cannot be incremented, while in the
second trace source only the exponents three and five
can be incremented. This gives its four displayed orbits.
For the other columns one uses exactly
\(u_i\subseteq_2 a_i\) from Lemma~\ref{lem:lucas}.
The two surviving cyclic orbits after addition are
\((3,5,7)\) and \((3,7,5)\), whose sum is \(m_{7,5,3}\).
This proves \eqref{eq:fifteen-cartan}.

The dual Cartan formula \eqref{eq:dual-cartan} gives
\[
\begin{array}{c|ccc}
 &\Sq_*^2&\Sq_*^4&\Sq_*^6\\
\hline
b_5^3&b_5^2b_3&b_5b_3^2&b_3^3\\
b_6^2b_3&b_5^2b_3&0&b_3^3\\
b_7b_5b_3&b_7b_3^2&0&0\\
b_9b_3^2&b_7b_3^2&b_5b_3^2&0
\end{array}
\]
The odd down squares \(1,3,5,7\) of each displayed
summand vanish. This follows from the odd-degree
vanishing on a square, from the spike exponents three
and seven, and from
\(\Sq_*^r b_j=\binom{j-r}{r}b_{j-r}\) for \(j=5,9\).
All indices at least eight vanish on degree fifteen
by adjoint instability.
Every column in the table sums to zero in \(\F\).
Thus \(\Theta_{15}\in K_{15}^\Sigma(3)\), and
\(\Theta_{15}(m_{7,5,3})=1\).
Equation~\eqref{eq:fifteen-cartan} and
Corollary~\ref{cor:relative-detection} prove the nonzero
relative class in three variables.
Proposition~\ref{prop:support} extends this class to
every \(n\geq3\).

By Lemma~\ref{lem:naturality-spikes}, every ordered monomial
of degree \(15=2^4-1\) is non-hit.
Finally the exact rank table in
Proposition~\ref{prop:auxiliary-initial-degrees} gives \(\dim\Crel_{15}(4)=1\),
so the nonzero class just constructed is a basis.
\end{proof}

\textit{Remark.} Theorem~\ref{thm:fifteen} shows that universal surjectivity
of \(\rho_d\) is false. It does not contradict the
monomial implication, since none of the ordered monomials
in the example satisfies its hypothesis.
The related degree-sixteen example \(m_{8,7,1}\) is
established in \cite[Example~25.2.2]{WalkerWoodII}.
The support splitting gives its representative
\(m_{8,7,1,0}\) in four variables.

\begin{theorem}\label{thm:normalizer-eighteen}
Let \(\tau\) act through the normalizer quotient of
Proposition~\ref{prop:normalizer}.
Then \([u],[v]\) form a basis of \(\Crel_{18}(4)\), and
\[
 [u]\tau=[u],\qquad [v]\tau=[v]+[u].
\]
Consequently \(\Crel_{18}(4)\cong\F C_2\) is indecomposable.
\end{theorem}
\begin{proof}
We first give ordinary hit expressions for the two representatives.
For a polynomial in three variables, let $\mathcal T$ be the cyclic
trace defined in Section~\ref{sec:prelim}, and put
\[
\begin{aligned}
 F_1={}&\mathcal T(xy^6z^{10}+x^3y^5z^9+x^3y^7z^7
                  +x^3y^8z^6+x^5y^5z^7+x^5y^6z^6),\\
 F_2={}&\mathcal T(x^2y^5z^9+x^3y^6z^7+x^3y^7z^6),\\
 F_4={}&\mathcal T(x^3y^5z^6).
\end{aligned}
\]
The complete nonzero cyclic coefficients of their squares are as follows;
a triple denotes its cyclic orbit, as in the degree-fifteen calculation.
\[
\begin{array}{c|ccc}
(a,b,c)&\Sq^1F_1&\Sq^2F_2&\Sq^4F_4\\\hline
(2,6,10)&1&1&0\\
(3,5,10)&1&0&1\\
(3,6,9)&1&1&0\\
(3,7,8)&1&1&0\\
(3,8,7)&1&1&0\\
(3,9,6)&0&1&1\\
(4,5,9)&1&1&0\\
(4,6,8)&0&1&1\\
(4,7,7)&1&0&0\\
(4,8,6)&1&1&0\\
(5,5,8)&1&0&1\\
(5,6,7)&1&1&0\\
(5,7,6)&1&1&0\\
(6,6,6)&1&0&1
\end{array}
\]
Each entry follows from~\eqref{eq:cartan}. Thus
$m_{7,7,4}=\Sq^1F_1+\Sq^2F_2+\Sq^4F_4$.
Applying the four support embeddings $\eta_j$ and summing gives
\[
 u=\sum_{r\in\{1,2,4\}}\Sq^r\left(\sum_{j=1}^4\eta_j(F_r)\right).
\]
Here each monomial of $u$ occurs once, because its support has
exactly three elements.

For the second representative, define the cyclic trace on the first
three variables by
\[
 \mathcal T_4(f)=f(x_1,x_2,x_3,x_4)+f(x_2,x_3,x_1,x_4)
                         +f(x_3,x_1,x_2,x_4).
\]
Let the finite sets of ordered tuples be
\[
\begin{aligned}
\mathcal E_1={}&\{(1,2,5,9),(1,2,7,7),(1,5,4,7),(1,5,7,4),\\
 &\quad(1,6,6,4),(1,7,4,5),(1,7,5,4),(2,3,9,3),\\
 &\quad(2,5,5,5),(2,7,7,1),(2,9,5,1),(2,10,4,1),\\
 &\quad(3,5,6,3),(4,5,5,3)\},\\
\mathcal E_2={}&\{(1,2,6,7),(1,2,7,6),(1,4,6,5),(2,2,7,5),\\
 &\quad(2,4,3,7),(2,5,6,3),(2,5,8,1),(2,6,7,1),\\
 &\quad(2,7,4,3),(2,7,6,1),(2,9,3,2),(3,4,7,2),\\
 &\quad(3,5,5,3),(3,6,4,3),(3,6,5,2),(3,7,4,2)\},\\
\mathcal E_4={}&\{(1,2,5,6),(2,3,4,5),(2,3,6,3),(2,4,4,4),\\
 &\quad(2,5,5,2),(2,6,5,1),(3,4,5,2)\},\\
\mathcal E_8={}&\{(2,2,3,3)\}.
\end{aligned}
\]
Set
\[
\begin{aligned}
 H_1={}&m_{6,5,5,1}+m_{7,5,4,1}+m_{10,3,2,2}
                      +\mathcal T_4\Bigl(\sum_{a\in\mathcal E_1}x^a\Bigr),\\
 H_2={}&m_{9,3,2,2}
                      +\mathcal T_4\Bigl(\sum_{a\in\mathcal E_2}x^a\Bigr),\\
 H_4={}&\mathcal T_4\Bigl(\sum_{a\in\mathcal E_4}x^a\Bigr),\qquad
 H_8=\mathcal T_4\Bigl(\sum_{a\in\mathcal E_8}x^a\Bigr).
\end{aligned}
\]
Then the exact identity
\begin{equation}\label{eq:eighteen-v-hit}
 v=\Sq^1H_1+\Sq^2H_2+\Sq^4H_4+\Sq^8H_8
\end{equation}
is verified by~\eqref{eq:cartan}: the surviving ordered monomials
are precisely the four, twelve, and six distinct terms of
$m_{5,5,5,3}$, $m_{7,4,4,3}$, and $m_{7,7,2,2}$, respectively.
In particular, both $u$ and $v$ are ordinary hits.

We verify their independence explicitly.
Using \eqref{eq:dual-cartan}, and grouping equal terms
before their final cancellation, gives
\[
\begin{aligned}
 \Sq_*^1\Theta_u&=
 2(b_7^2b_3+b_9b_5b_3+b_9b_7b_1+b_{11}b_3^2+b_{11}b_5b_1),\\
 \Sq_*^2\Theta_u&=
 2(b_7^2b_2+b_7b_6b_3+b_{10}b_3^2+b_{11}b_3b_2),\\
 \Sq_*^4\Theta_u&=
 2(b_7b_4b_3+b_6b_5b_3+b_7b_5b_2+b_7b_6b_1),\\
 \Sq_*^8\Theta_u&=0,
\end{aligned}
\]
and
\[
\begin{aligned}
 \Sq_*^1\Theta_v&=2b_7^2b_2b_1,\\
 \Sq_*^2\Theta_v&=2b_7^2b_1^2+2b_{11}b_3b_1^2,\\
 \Sq_*^4\Theta_v&=2b_7b_5b_1^2,\qquad
 \Sq_*^8\Theta_v=0.
\end{aligned}
\]
Every displayed expression is zero in \(\F\).
The power-of-two squares of index at least sixteen
vanish by adjoint instability. By
Lemma~\ref{lem:power-generators}, these squares generate
the Steenrod algebra, so \(\Theta_u,\Theta_v\) annihilate
all symmetric hits. Their coefficients give
\[
 \begin{pmatrix}\Theta_u(u)&\Theta_u(v)\\
                 \Theta_v(u)&\Theta_v(v)\end{pmatrix}
   =\begin{pmatrix}1&0\\0&1\end{pmatrix}.
\]
The two relative classes are therefore independent.

The degree-eighteen rank calculation gives
\[
 b_{18}=84,\quad r_{18}=1204,\quad s_{18}=76,\quad a_{18}=1210.
\]
Proposition~\ref{prop:rank-formula} then gives
\(\dim\Crel_{18}(4)=84-76-(1210-1204)=2\), so
the two classes form a basis.
Since $\tau$ preserves the relative space, the two dual detectors already
proved Steenrod-annihilated determine its action on the basis $[u],[v]$.
For completeness, the coefficients needed for their evaluation are
\[
\begin{array}{c|cc@{\qquad}c|cc}
 a&\epsilon_a(u\tau)&\epsilon_a(v\tau)
 &a&\epsilon_a(u\tau)&\epsilon_a(v\tau)\\\hline
(7,7,4,0)&1&0&(7,7,2,2)&0&1\\
(8,7,3,0)&1&1&(9,7,1,1)&0&0\\
(9,6,3,0)&0&0&(11,5,1,1)&0&1\\
(9,7,2,0)&0&0&(13,3,1,1)&0&1\\
(10,5,3,0)&0&1&&&\\
(10,7,1,0)&0&0&&&\\
(11,4,3,0)&0&0&&&\\
(11,5,2,0)&1&0&&&\\
(11,6,1,0)&1&1&&&\\
(12,3,3,0)&1&0&&&
\end{array}
\]
Indeed, the coefficient at a fixed ordered target $x^a$ of
$(x^b)\tau$ is the coefficient of $x^a$ in
\[
 \prod_{i=1}^4\left(\sum_{j\ne i}x_j\right)^{b_i}
 =\prod_{i=1}^4\prod_{k:b_{ik}=1}
                      \left(\sum_{j\ne i}x_j^{2^k}\right),
 \qquad b_i=\sum_k b_{ik}2^k.
\]
Adding this coefficient over the distinct terms of $u$ or $v$
gives each table entry. The left and right blocks are exactly the
supports of $\Theta_u$ and $\Theta_v$, with zero parts appended.
Their column sums in $\F$ give
\[
 \begin{pmatrix}
 \Theta_u(u\tau)&\Theta_u(v\tau)\\
 \Theta_v(u\tau)&\Theta_v(v\tau)
 \end{pmatrix}
 =\begin{pmatrix}1&1\\0&1\end{pmatrix}.
\]
The identity pairing of $\Theta_u,\Theta_v$ with $[u],[v]$
therefore yields $[u]\tau=[u]$ and $[v]\tau=[v]+[u]$.

Since \((\tau-1)[v]=[u]\ne0\) and
\((\tau-1)[u]=0\), the action has one Jordan block
of size two. Sending the two group elements \(1,\tau\)
of the regular module to \([v],[v]+[u]\) gives the
claimed isomorphism. A decomposition into two
one-dimensional modules would have trivial
\(\tau\)-action and is therefore impossible.
\end{proof}

\section{Kameko maps and classifications by binary columns}
\label{sec:kameko}

The up and down maps relate ordinary cohits and permutation
invariants in the same degree change. Their compatibility
with inclusion gives relative isomorphisms. We prove the
compatibility explicitly and retain the normalizer action.

\subsection{Relative and dual Kameko maps}

\begin{proposition}\label{prop:kameko-relative}
Let \(n\geq1\), \(d\geq0\), \(D=2d+n\), and \(H\leq\Sigma_n\).
Then \(\kappa_n\) induces a map
\(\Crel_D(H)\to\Crel_d(H)\), equivariant for
\(N_{G_n}(H)/H\). If \(\mu(D)=n\), then this map is an
isomorphism whose inverse is induced by \(\upsilon_n\).
\end{proposition}
\begin{proof}
The same polynomial map induces both vertical maps in the commuting
square
\[
\begin{array}{ccc}
 Q_D(P(n)^H)&\longrightarrow&Q_D(P(n))\\
 \big\downarrow{\kappa_n}&&\big\downarrow{\kappa_n}\\
 Q_d(P(n)^H)&\longrightarrow&Q_d(P(n)).
\end{array}
\]
The horizontal maps are induced by inclusion. The kernels of these
maps are precisely \(\Crel_D(H)\) and \(\Crel_d(H)\), by
\eqref{eq:relative} and the definition of \(\Crel_d(H)\).
Commutativity therefore gives the asserted map on kernels.

If \(\mu(D)=n\), both vertical arrows are isomorphisms by
Proposition~\ref{prop:kameko-cohit}. To check surjectivity on the
kernels explicitly, take a class in the lower left kernel and lift
it by the left inverse isomorphism. Its image in the upper right
space maps to zero under the right vertical isomorphism, so that
image is zero. The lift thus belongs to the upper left kernel.
Injectivity is inherited from the left vertical isomorphism. The
inverse on these kernels is induced by \(\upsilon_n\), since it is
the inverse of each vertical map.

An element of \(N_{G_n}(H)\) preserves \(P(n)^H\), both hit spaces,
and their inclusion. Proposition~\ref{prop:kameko-cohit} shows that
\(\kappa_n\) commutes with this action. The subgroup \(H\) acts
trivially on its fixed polynomials, so the action factors through
\(N_{G_n}(H)/H\). This proves the equivariance, including that of
the inverse when the numerical hypothesis holds.
\end{proof}

\begin{proposition}\label{prop:dual-kameko}
Let \(n\geq1\), \(d\geq0\), and \(D=2d+n\).
Then \(\kappa_{n,B}^*\) induces a map
\(\Rel_d(n)\to\Rel_D(n)\), which is the transpose
of the relative down map under \eqref{eq:duality}.
If \(\mu(D)=n\), then this map is an isomorphism.
For a partition \(\lambda\) with \(r\leq n\) positive
parts, its value on the symmetric dual basis is
\[
 \kappa_{n,B}^*(b_\lambda)
   =b_{2\lambda_1+1}\cdots b_{2\lambda_r+1}b_1^{\,n-r}.
\]
\end{proposition}
\begin{proof}
An element of \(K_d^\Sigma(n)\) annihilates every
symmetric hit. Its pullback along \(\kappa_n\) therefore
annihilates every symmetric hit in degree \(D\), by
\eqref{eq:kameko-squares}. The same argument applies
to the ordinary kernels. Since restriction and
pullback along \(\kappa_n\) commute, the image
\(\rho_d(K_d(n))\) maps into \(\rho_D(K_D(n))\).
This gives the asserted quotient map.
The pairing with a relative representative \(f\) is
\(\langle\Theta,\kappa_n(f)\rangle\), which proves
the transpose assertion. If \(\mu(D)=n\),
Proposition~\ref{prop:kameko-relative} and finite
dimensionality give the isomorphism.

The symmetric down map sends \(m_\eta\) to zero unless
all \(n\) exponents, including zero parts, are odd.
When they are odd, it sends it to
\(m_{(\eta-\mathbf1)/2}\). There is exactly one orbit
with this image equal to \(m_\lambda\): its positive
parts are \(2\lambda_i+1\), together with \(n-r\)
parts equal to one. Evaluating on the basis gives
the displayed formula.
\end{proof}

Walker and Wood's stable symmetric down map preserves
length \(r\) and changes degree \(2d+r\) to \(d\);
its dual sends \(b_j\) to \(b_{2j+1}\)
\cite[Definition~26.1.9 and Propositions~26.1.10--26.1.11]{WalkerWoodII}.
Proposition~\ref{prop:dual-kameko} is the fixed-\(n\)
version. The extra \(b_1^{n-r}\) accounts for the
zero parts in a polynomial in \(n\) variables.
These degree-changing dual maps are distinct from
the identity operation \(\Sq^0\).

\subsection{Numerical reduction and complete bases}

The optimal bound for ordinary Kameko isomorphisms is given by
T\'in and Sum \cite[Theorem~3]{TinSum2016}.
For the invariant and relative applications, we use the following
binary reformulation of the numerical criterion
\cite[Proposition~2.4.4]{WalkerWoodI}. It specifies the individual
degree changes for which \(\mu(D_t^{(n)}(e))=n\).

\begin{lemma}\label{lem:kameko-binary-criterion}
Let $n\geq2$, $e\geq0$, and $t\geq1$. Then
\[
 \mu(D_t^{(n)}(e))=n
 \quad\Longleftrightarrow\quad
 \alpha(e+n)\leq n
 \quad\hbox{and}\quad
 \alpha(e+n-1)+t-1>n-2.
\]
If $\alpha(e+n)\leq n$, then
$\mu(D_t^{(n)}(e))=n$ for every $t\geq s_n(e)+1$.
\end{lemma}
\begin{proof}
Put \(D=D_t^{(n)}(e)\). Since \(D>0\), \(D\equiv n\pmod2\),
and positive spike exponents are odd, \(\mu(D)\equiv n\pmod2\).
By \eqref{eq:mu}, the condition \(\mu(D)=n\) is therefore
equivalent to \(\alpha(D+n)\leq n\) and
\(\alpha(D+n-2)>n-2\). Now
\[
 \alpha(D+n)=\alpha(e+n),\qquad
 \alpha(D+n-2)=\alpha(e+n-1)+t-1,
\]
where the second equality follows from
\(D+n-2=(e+n-1)2^t+(2^t-2)\).
These identities give the equivalence. The second inequality,
together with \(t\geq1\), is exactly \(t\geq s_n(e)+1\).
\end{proof}

\begin{corollary}\label{lem:kameko-numerical}
Let \(e\geq0\). Then \(\mu(\Delta_t(e))=4\) for every \(t\geq1\)
if and only if \(\alpha(e+4)\leq4\) and \(\alpha(e+3)\geq3\).
\end{corollary}
\begin{proof}
Apply Lemma~\ref{lem:kameko-binary-criterion} with \(n=4\).
Its second inequality is weakest at \(t=1\), where it is
\(\alpha(e+3)>2\). This is equivalent to the stated integral
inequality.
\end{proof}

For a positive integer \(m\), let \(\nu_2(m)\) denote the largest
integer \(v\geq0\) such that \(2^v\) divides \(m\). The identity
\(\alpha(m-1)=\alpha(m)-1+\nu_2(m)\) gives
\begin{equation}\label{eq:optimal-kameko-bound}
 s_n(e)=\max\{0,n-\alpha(e+n)-\nu_2(e+n)\}.
\end{equation}
Thus \(s_n(e)\) is the least starting index for ordinary Kameko
isomorphisms, as established by T\'in and Sum
\cite[Theorem~3]{TinSum2016}. The uniform upper bound
\(s_n(e)\leq n-2\) and its sharpness are recorded in
\cite[Corollary~4 and the subsequent discussion]{TinSum2016}.
The next proposition proves that the same starting index is least
for the cohits of every permutation-invariant algebra and gives
the induced relative isomorphisms.

\begin{proposition}\label{prop:eventual-kameko-stability}
Let \(n\geq2\), \(H\leq\Sigma_n\), \(e\geq0\), and
\(s=s_n(e)\). If \(\alpha(e+n)>n\), then \(s=0\) and
\[
 Q_{D_t^{(n)}(e)}(P(n)^H)=0
 \qquad(t\geq0).
\]
If \(\alpha(e+n)\leq n\), then, for every \(t\geq s\),
the iterated down map \(\kappa_n^{t-s}\) induces isomorphisms
\[
\begin{aligned}
 Q_{D_t^{(n)}(e)}(P(n)^H)
    &\xrightarrow{\ \cong\ } Q_{D_s^{(n)}(e)}(P(n)^H),\\
 \Crel_{D_t^{(n)}(e)}(H)
    &\xrightarrow{\ \cong\ } \Crel_{D_s^{(n)}(e)}(H).
\end{aligned}
\]
The relative isomorphism is equivariant for \(N_{G_n}(H)/H\).
Moreover, for every integer \(j\) with \(0\leq j<s\),
\[
 \kappa_n:
 Q_{D_{j+1}^{(n)}(e)}(P(n)^H)
 \longrightarrow Q_{D_j^{(n)}(e)}(P(n)^H)
\]
is not injective. Thus \(s\) is the least nonnegative integer
from which all successive Kameko maps on these invariant cohit
spaces are isomorphisms. One has \(0\leq s\leq n-2\), and
the upper bound is attained for every \(n\geq2\).
\end{proposition}
\begin{proof}
Put \(m=e+n\). For every \(t\geq0\),
\[
 \alpha(D_t^{(n)}(e)+n)=\alpha(m2^t)=\alpha(m).
\]
If \(\alpha(m)>n\), \eqref{eq:optimal-kameko-bound} gives \(s=0\),
while \eqref{eq:mu} gives \(\mu(D_t^{(n)}(e))>n\).
Proposition~\ref{prop:classical} therefore gives the asserted
vanishing in every degree of the sequence.

Assume now that \(\alpha(m)\leq n\).
Lemma~\ref{lem:kameko-binary-criterion} gives
\(\mu(D_i^{(n)}(e))=n\) for every \(i\geq s+1\).
Since
\[
 D_i^{(n)}(e)=2D_{i-1}^{(n)}(e)+n,
\]
Propositions~\ref{prop:kameko-cohit} and
\ref{prop:kameko-relative} apply successively for
\(i=t,t-1,\ldots,s+1\). Their composition gives the
asserted isomorphisms, with inverses induced by
\(\upsilon_n^{t-s}\). If \(t=s\), these maps are the identity.
Each relative map is equivariant for \(N_{G_n}(H)/H\), so
their composition is equivariant as well.

To prove the claimed optimality for invariant cohits, let
\(0\leq j<s\), and put \(D=D_{j+1}^{(n)}(e)\).
The exact criterion in
Lemma~\ref{lem:kameko-binary-criterion} gives
\(\mu(D)\ne n\), whereas
\(\alpha(D+n)=\alpha(m)\leq n\) gives \(\mu(D)\leq n\).
Therefore \(\mu(D)<n\). In fact \(\mu(D)\leq n-2\),
because \(D\) and \(n\) have the same parity, as does
the number of positive spike summands of \(D\).

Choose a minimal spike \(x^a\) of degree \(D\) in \(n\)
variables and form its \(H\)-orbit sum
\[
 b=N_{H\cap H_a}^{H}(x^a)\in P^D(n)^H.
\]
The coefficient functional of the spike \(x^a\) evaluates
to one on \(b\) and annihilates all ordinary hits by
Lemma~\ref{lem:naturality-spikes}.
It follows that \([b]\ne0\) in \(Q_D(P(n)^H)\).

Every monomial in \(b\) has exactly \(\mu(D)<n\)
positive exponents. In particular, each has a zero exponent,
and therefore \(\kappa_n(b)=0\). This nonzero cohit class
lies in the kernel of the displayed Kameko map, proving
noninjectivity for every \(j<s\). Together with the
isomorphisms already established, this proves the minimality
of \(s\) for the invariant cohit spaces.

The final numerical assertion is the uniform bound and sharpness
result recalled before the proposition.
\end{proof}

The minimality assertion concerns the invariant cohit spaces.
The relative spaces can stabilize earlier; for example, the family
\(7\cdot2^t-4\) has zero relative space at every index, although
\(s_4(3)=1\), as proved in Theorem~\ref{thm:seven-generic-family}.

The isomorphisms determine more than dimensions. They give a complete hit basis through the
following decomposition of the orbit basis.

\begin{proposition}\label{prop:transport-complete-bases}
Let $n\geq2$, $E,u\geq0$, and $D=(E+n)2^u-n$. Suppose that
$\mu((E+n)2^j-n)=n$ for $1\leq j\leq u$.
Let $\mathcal H_E$ be a basis of $\Hit^E(B(n))$, and suppose that
the classes of the polynomials in $\mathcal Q_E$ form a basis of
$Q_E(B(n))$. Then
\[
 \{m_\lambda:\lambda\in\mathscr Z_{E,u}^{(n)}\}
 \ \cup\ \{\upsilon_n^u(h):h\in\mathcal H_E\}
\]
is a basis of $\Hit^D(B(n))$, and
\[
 \{[\upsilon_n^u(q)]:q\in\mathcal Q_E\}
\]
is a basis of $Q_D(B(n))$. In particular,
\[
 \dim\Hit^D(B(n))
   =|\mathscr P_D^{(n)}|-\dim Q_E(B(n)).
\]
\end{proposition}
\begin{proof}
The map $a\mapsto a^{[u]}$ is injective on
$\mathscr P_E^{(n)}$ because each coordinate can be recovered by
$a_i=(a_i^{[u]}+1)/2^u-1$. Its image consists precisely of the
partitions $\lambda\in\mathscr P_D^{(n)}$ whose coordinates all satisfy
$\lambda_i\equiv2^u-1\pmod{2^u}$. Indeed, these congruences make
$a_i=(\lambda_i+1)/2^u-1$ nonnegative integers; the order of the
coordinates is preserved, and summing gives $|a|=E$.
For such a partition, $\upsilon_n^u(m_a)=m_{a^{[u]}}$.
The ambient orbit basis therefore gives a direct sum
\[
 B^D(n)
 =\operatorname{span}\{m_\lambda:\lambda\in\mathscr Z_{E,u}^{(n)}\}
   \ \oplus\ \upsilon_n^u(B^E(n)).
\]
The hypotheses and Proposition~\ref{prop:kameko-cohit} imply that
$\kappa_n^u:Q_D(B(n))\longrightarrow Q_E(B(n))$ is an isomorphism.
For $\lambda\in\mathscr Z_{E,u}^{(n)}$, at least one of the first
$u$ binary columns fails to consist entirely of ones. Successive
applications of $\kappa_n$ therefore give
$\kappa_n^u(m_\lambda)=0$. Injectivity on cohits implies that
$m_\lambda\in\Hit^D(B(n))$.
Also $\upsilon_n^u$ induces the inverse isomorphism on cohits, so it
sends every element of $\Hit^E(B(n))$ to $\Hit^D(B(n))$.

Conversely, let $f\in\Hit^D(B(n))$. The displayed ambient direct sum
gives unique polynomials $z$ and $g$ with
\[
 f=z+\upsilon_n^u(g),\qquad
 z\in\operatorname{span}\{m_\lambda:\lambda\in\mathscr Z_{E,u}^{(n)}\},
 \quad g\in B^E(n).
\]
The preceding paragraph shows that $z$ is hit, so
$[\upsilon_n^u(g)]=[f]+[z]=0$ in $Q_D(B(n))$.
Applying the down-map isomorphism gives $[g]=0$, hence
$g\in\Hit^E(B(n))$. Thus
\[
 \Hit^D(B(n))
 =\operatorname{span}\{m_\lambda:\lambda\in\mathscr Z_{E,u}^{(n)}\}
   \ \oplus\ \upsilon_n^u(\Hit^E(B(n))).
\]
The first summand has its displayed orbit basis. The map
$\upsilon_n^u$ is injective on polynomials, and therefore sends the
basis $\mathcal H_E$ to a basis of the second summand. The two bases
are disjoint and linearly independent because the sum is direct.
This proves the hit-basis assertion. The quotient-basis assertion
follows by applying the inverse Kameko isomorphism to the given basis
of $Q_E(B(n))$. Counting the ambient orbit basis and subtracting the
quotient dimension gives the final formula.
\end{proof}

\subsection{Classification by binary columns}

The preceding decomposition also reduces the ordinary monomial
test and the symmetric orbit-sum test simultaneously.

\begin{theorem}\label{thm:binary-classification}
Let \(e\geq0\) satisfy
\(\alpha(e+4)\leq4\) and \(\alpha(e+3)\geq3\).
Let \(t\geq0\), and let \(\lambda\in\mathbb N^4\)
have degree \(\Delta_t(e)\).
Then
\[
 \Crel_{\Delta_t(e)}(4)\cong\Crel_e(4)
\]
as modules for \(N_{G_4}(H_4)/H_4\).
If some \(\lambda_i\not\equiv2^t-1\pmod{2^t}\),
then \(x^\lambda\) and \(m_\lambda\) are ordinary
and symmetric hits, respectively.
If all four congruences hold, then
\(\lambda=a^{[t]}\) for a unique \(a\in\mathbb N^4\)
of degree \(e\), its stabilizer equals that of \(a\),
and
\[
 \begin{aligned}
 x^\lambda\in\Hit(P(4))&\ \Longleftrightarrow\
               x^a\in\Hit(P(4)),\\
 m_\lambda\in\Hit(B(4))&\ \Longleftrightarrow\
               m_a\in\Hit(B(4)).
 \end{aligned}
\]
\end{theorem}
\begin{proof}
For \(t=0\) every assertion follows from
\(\Delta_0(e)=e\).
For \(t>0\), Corollary~\ref{lem:kameko-numerical} gives
\(\mu(\Delta_j(e))=4\) for \(1\leq j\leq t\).
Since \(\Delta_j(e)=2\Delta_{j-1}(e)+4\),
Proposition~\ref{prop:kameko-relative} can be
iterated, giving the equivariant relative isomorphism.

Suppose first that all congruences hold.
The least nonnegative integer in the required residue
class is \(2^t-1\). Therefore
\(a_i=(\lambda_i+1)/2^t-1\) is a nonnegative integer.
Summing gives
\[
 2^t(|a|+4)-4=\Delta_t(e),
\]
so \(|a|=e\), and the formula gives uniqueness.
The map \(a_i\mapsto2^t(a_i+1)-1\) preserves and
reflects equality of exponents, proving the stabilizer
assertion. It also gives
\(\upsilon_4^t(x^a)=x^\lambda\) and
\(\upsilon_4^t(m_a)=m_\lambda\).
Every constituent up map is an isomorphism on both
cohit spaces by Proposition~\ref{prop:kameko-cohit},
so either image class is zero exactly when its
initial class is zero. This proves the two equivalences.

Suppose instead that at least one congruence fails.
Let \(j<t\) be the first binary column that is not
all ones, counting the lowest column as column zero.
Removing the preceding \(j\) all-one columns gives
a tuple \(b\) of degree \(\Delta_{t-j}(e)\) with
an even exponent, and \(\lambda=b^{[j]}\).
Its first weight entry is less than four, while
\(\mu(\Delta_{t-j}(e))=4\).
Proposition~\ref{prop:classical} makes both \(x^b\)
and \(m_b\) hit. The successive well-defined up maps
send their zero classes to zero classes.
This proves both hit assertions for \(\lambda\).
\end{proof}

\subsection{Stabilizers under successive degree changes}

The exponent transformation preserves equality of coordinates.
The fixed-preimage constructions therefore persist with their
stabilizers under the same isomorphisms.

\begin{theorem}\label{thm:two-pair-family}
Let \(t\geq0\).
Then \(A_t\) and \(B_t\) have stabilizer
\(\langle(12),(34)\rangle\cong\Sigma_2\times\Sigma_2\).
The ordered monomials \(x^{A_t}\) and \(x^{B_t}\) admit Steenrod preimages fixed
by this subgroup, both orbit sums are symmetrically
hit, and
\[
 \omega(A_t)=\omega(B_t)>_l
                 \omega_{\min}(74\cdot2^t-4).
\]
\end{theorem}
\begin{proof}
The case \(t=0\) is Proposition~\ref{prop:pairpair-seventy}.
Here \(A_t=A_0^{[t]}\), \(B_t=B_0^{[t]}\), and
the common degree is \(\Delta_t(70)\).
Since \(\alpha(74)=\alpha(73)=3\),
Corollary~\ref{lem:kameko-numerical} gives
\(\mu(\Delta_j(70))=4\) for every \(j\geq1\).
Apply Proposition~\ref{prop:kameko-cohit},
first with \(H=\langle(12),(34)\rangle\), then with
\(H=\Sigma_4\). The fixed-preimage and symmetric-hit
properties of the initial targets both propagate.
The strictly increasing exponent transformation
preserves the two equality pairs and no other
equalities, proving the stabilizer assertion.

The weights are obtained by placing \(t\) entries
equal to four before \((2,2,4,2,2)\).
For \(t\geq1\) the minimal spike is
\[
 (2^{t+6}-1,2^{t+3}-1,2^t-1,2^t-1).
\]
Its four exponents sum to \(74\cdot2^t-4\);
their spike lengths have the decreasing order
required of the minimal spike. Its weight is
\(t\) entries equal to four followed by
\((2,2,2,1,1,1)\).
The first differing entry is consequently \(4>2\).
The same comparison for \(t=0\) was proved in
Proposition~\ref{prop:pairpair-seventy}.
\end{proof}

\begin{corollary}\label{cor:lambda-classification}
Let \(s\geq3\) and \(t\geq0\).
Then \(x^{\lambda_s^{[t]}}\) is ordinary hit,
\(m_{\lambda_s^{[t]}}\) is symmetrically hit,
its stabilizer is \(\Sigma_2\), and
\[
 \omega(\lambda_s^{[t]})>_l
          \omega_{\min}(\Delta_t(2^{s+1}+6)).
\]
If \(s=3\), then every symmetric ordinary hit
of the same degree is a symmetric hit.
\end{corollary}
\begin{proof}
For \(s\geq4\), apply Proposition~\ref{prop:one-square}
with \(q=2^{s-1}\), \(a=1\), \(b=3\), \(c=1\).
The required inequality is \(6<q\).
It proves both initial hit assertions.
For \(s=3\), the exact symmetric expansion is
\[
 m_{9,9,3,1}=\Sq^8m_{5,5,3,1}+\Sq^1m_{10,6,4,1}.
\]
In the first square the possible target partitions
and their integer multiplicities are
\[
 \begin{array}{c|rrrrrr}
 \text{target}&(10,6,5,1)&(10,6,4,2)&(10,5,5,2)&
              (9,9,3,1)&(9,6,6,1)&(9,6,5,2)\\ \hline
 \text{multiplicity}&2&1&2&1&2&2 .
 \end{array}
\]
Lemma~\ref{lem:lucas} allows increments \(0,1,4,5\)
in the two exponent-five positions, \(0,1,2,3\)
in the exponent-three position, and \(0,1\) in the
last position. Their combinations summing to eight
give exactly this table by \eqref{eq:coefficient-test}.
Only the last source exponent in \(m_{10,6,4,1}\)
is odd, so its first square is \(m_{10,6,4,2}\).
This proves the displayed identity.
The ordered target factors as
\[
 x_1^9x_2^9x_3^3x_4
 =(x_1x_2x_3^3x_4)(x_1x_2)^8.
\]
The low factor has degree six, whereas
\(6<7\mu(2)=14\). The Silverman--Singer criterion
\cite[Theorem~14.1.2]{WalkerWoodI}, used in
Proposition~\ref{prop:two-square}, proves
that it is ordinary hit.

The initial degree is \(e=2^{s+1}+6\), and
\(\alpha(e+4)=\alpha(e+3)=3\).
Theorem~\ref{thm:binary-classification} propagates
both hit assertions and the stabilizer.
The initial minimal spike is
\((2^{s+1}-1,7,0,0)\); its first weight entry is two,
whereas all four target exponents are odd.
At the \(t\)-th step the minimal spike is
\[
 (2^{s+t+1}-1,2^{t+3}-1,2^t-1,2^t-1).
\]
Its weight and that of the target agree in the first
\(t\) entries, all equal to four, and then have entries
two and four, respectively.
This proves the strict inequality.
For \(s=3\), the initial degree is \(22\);
Corollary~\ref{cor:degree-families} gives vanishing
of the full relative space in every lifted degree.
\end{proof}

\begin{corollary}\label{cor:two-square-lifts}
Let \(a\geq3\), \(a\equiv3\pmod4\), and
\(\alpha(a+5)=2\). Let \(t\geq0\).
Then \(x^{\theta(a)^{[t]}}\) and
\(m_{\theta(a)^{[t]}}\) are ordinary and symmetric
hits, respectively, with stabilizer \(\Sigma_2\),
and their weight is strictly greater than
\(\omega_{\min}(\Delta_t(2a+8))\).
\end{corollary}
\begin{proof}
Write \(a+5=2^r+2^s\), where \(r>s\geq2\);
the bound \(s\geq2\) follows from \(4\mid(a+5)\).
For \(e=2a+8\),
\[
 e+4=2^{r+1}+2^{s+1}+2,\qquad
 e+3=2^{r+1}+2^{s+1}+1.
\]
Both digit sums equal three.
Proposition~\ref{prop:two-square} proves the initial
hit assertions and \(\mu(e)=2\).
Theorem~\ref{thm:binary-classification} propagates
membership and the stabilizer.
The initial minimal spike has lengths \(r+1,s+1\);
after \(t\) steps its four lengths are
\(r+t+1,s+t+1,t,t\).
The target has \(t+1\) initial all-one binary columns,
while the minimal spike has \(t\) such columns
followed by a column of weight two.
This proves the asserted strict comparison.
\end{proof}

\begin{remark}\label{rem:two-square-boundary}
The condition \(\alpha(a+5)=2\) cannot be weakened to
\(\alpha(a+5)\leq2\) in Corollary~\ref{cor:two-square-lifts}.
Assume \(a\geq3\) and \(a\equiv3\pmod4\), and suppose
that \(a+5=2^r\). Then \(r\geq3\), \(a\equiv3\pmod8\),
and \(\theta(a)=(2^r+1,2^r-5,1,1)\). Its first lift is
\[
 \theta(a)^{[1]}=(2^{r+1}+3,2^{r+1}-9,3,3),
\]
which has degree \(2^{r+2}\) and four odd exponents.
Proposition~\ref{prop:power-of-two-monomials} makes this
ordered monomial non-hit, although the initial monomial is
hit by Proposition~\ref{prop:two-square}.
For \(E=2^{r+2}\), one has \(\alpha(E+4)=2\) and
\(\alpha(E+3)=3\). Theorem~\ref{thm:binary-classification}
therefore shows that all further lifted ordered monomials remain
non-hit. The initial equality \(\alpha(a+5)=1\) is thus a
genuine obstruction to extending the ordinary-hit assertion.
\end{remark}

\textit{Remark.} The classifications in Theorem~\ref{thm:binary-classification}
apply to every monomial in the specified degrees.
They include the adjacent initial degrees in
Corollary~\ref{cor:degree-families} and the explicit
above-minimal-spike families.
The general fixed-preimage problem for arbitrary
repeated exponents is not settled by these families.
Proposition~\ref{prop:descent} states the precise
remaining sufficient condition, while
Theorem~\ref{thm:fifteen} excludes replacing it by
universal relative vanishing.

\section{Explicit bases in generic degree families}
\label{sec:computation}

Section~\ref{sec:kameko} reduces the families \(d_t\) and \(e_t\)
to degrees \(1,6,16\) and \(3,10\), respectively. We compute these
finite initial cases, prove the explicit symmetric cohit bases,
and then apply the degree isomorphisms. The finite column reductions
also give complete symmetric hit bases, enabling
Proposition~\ref{prop:transport-complete-bases} to determine every
subsequent basis without further Steenrod matrix reductions.

\subsection{Reduction to finite-dimensional linear algebra}

The finite hit matrices need only the power-of-two Steenrod squares.
The classical algebra generation theorem, recorded in
\cite[Proposition~3.2.4]{WalkerWoodI}, says that
\(\Sq^{2^j}\), \(j\geq0\), generate \(\A\).
Its standard consequence for the hit submodule of a left module is
the following formula.

\begin{lemma}\label{lem:power-generators}
Let \(M\) be a left module over \(\A\). Then
\[
 \A^+M=\sum_{j\geq0}\Sq^{2^j}M.
\]
\end{lemma}

Together with the instability bound \(r\leq d/2\) fixed in
Section~\ref{sec:prelim}, this lemma justifies the finite matrices
\(O_d\) and \(S_d\).

\begin{proposition}\label{prop:rank-formula}
Let \(d\geq0\). Then
\[
 \dim Q_d(P(4))=N_d-r_d,\qquad
 \dim Q_d(B(4))=b_d-s_d,\qquad
 \rank\bar\iota_d=a_d-r_d,
\]
and
\begin{equation}\label{eq:rank-formula}
 c_d=b_d-s_d-a_d+r_d.
\end{equation}
The same formulas hold if all positive square indices are used.
\end{proposition}

\begin{proof}
Lemma~\ref{lem:power-generators} identifies the column spaces
of \(O_d\) and \(S_d\), in their respective bases, with
\(\Hit^d(P(4))\) and \(\Hit^d(B(4))\). Taking their
codimensions proves the first two formulas.

The column space of \([O_d\mid A_d]\) is
\(\Hit^d(P(4))+B^d(4)\). Its quotient by
\(\Hit^d(P(4))\) therefore has dimension \(a_d-r_d\).
Every element of this quotient is represented by a symmetric
polynomial, and every symmetric polynomial is a linear
combination of the columns of \(A_d\). Thus this quotient is
precisely the image of \(\bar\iota_d\), proving the third formula.
The definition \eqref{eq:relative} and the dimension formula
for the kernel of a linear map now give
\[
 c_d=\dim Q_d(B(4))-\rank\bar\iota_d
     =b_d-s_d-a_d+r_d.
\]
Using all positive squares produces the same hit spaces by
Lemma~\ref{lem:power-generators}, so it leaves all these
dimensions and ranks unchanged.
\end{proof}

\subsection{Successive stages of the calculation}
\label{subsec:computation-stages}

\par\medskip\noindent\textit{Selection of bases and square indices.}
For a fixed degree \(d\), the calculation first lists the target
sets \(\mathscr E_d\) and \(\mathscr P_d\). For each power of
two \(r\leq d/2\), it also lists the source sets
\(\mathscr E_{d-r}\) and \(\mathscr P_{d-r}\). Each ordered
tuple is included exactly once, and each symmetric source is
represented by its nonincreasing tuple. Consequently the full
ordinary and symmetric matrices have respectively
\[
 \sum_{\substack{r=2^j\\r\leq d/2}}N_{d-r}
 \quad\text{and}\quad
 \sum_{\substack{r=2^j\\r\leq d/2}}b_{d-r}
\]
columns. These counts include zero columns. Removing such
columns does not change a column space or its rank.

\par\medskip\noindent\textit{Construction of the ordinary hit matrix.}
The next stage computes the coordinates of each
\(\Sq^r(x^a)\). By the Cartan formula \eqref{eq:cartan}, the
entry of \(O_d\) in target row \(b\) and source column
\((r,a)\) is
\[
 \prod_{i=1}^4\binom{a_i}{b_i-a_i}.
\]
The entry is zero if some \(b_i<a_i\). Otherwise
Lemma~\ref{lem:lucas} determines each factor from the binary
digits of \(a_i\) and \(b_i-a_i\). Since the source and target
degrees are \(d-r\) and \(d\), their increments already have
total \(r\). This determines every entry of \(O_d\) exactly.

\par\medskip\noindent\textit{Construction of the symmetric and inclusion matrices.}
For a symmetric source \(m_\mu\), the calculation expands the
square of each distinct ordered monomial in its orbit and adds
the results in \(\F\). Its coefficient on \(m_\lambda\) is
\(\epsilon_\lambda(\Sq^r m_\mu)\), calculated by
\eqref{eq:coefficient-test}. The coefficient is read at the
single ordered target \(x^\lambda\). Indeed, equivariance of
the Steenrod action makes the output symmetric, so this
coefficient equals the coefficient at every monomial in the
same orbit. Summing those equal coefficients over the target
orbit would multiply the desired entry by the orbit size.
In characteristic two this would erase the entry whenever
that size is even.

The inclusion matrix is constructed by placing a one in
exactly the rows occupied by the distinct monomials of each
\(m_\lambda\). Thus each column of \(A_d\) is the ordinary
coordinate vector of its corresponding symmetric basis
element. The supports of these columns are disjoint, and
each support is nonempty, so \(A_d\) has rank \(b_d\).

\par\medskip\noindent\textit{Rank calculation and passage to cohits.}
Gaussian elimination over \(\F\) is applied separately to
\(O_d\), \(S_d\), and \([O_d\mid A_d]\). Column interchanges
and addition of one column to another preserve the column
space. The nonzero columns in column echelon form are
linearly independent and span the same space as the input
columns. Counting them therefore gives \(r_d\), \(s_d\),
and \(a_d\), respectively. Proposition~\ref{prop:rank-formula}
then gives the ordinary and symmetric cohit dimensions and
the relative dimension \(c_d\).

\par\medskip\noindent\textit{Verification of the monomial implication.}
For every \(\lambda\in\mathscr P_d\), the calculation first
tests whether the coordinate vector of \(x^\lambda\) belongs
to the column space of \(O_d\). Whenever it does, the
coordinate vector of \(m_\lambda\) in the symmetric basis is
tested against the column space of \(S_d\). Membership is
equivalent to having zero remainder after elimination.
Equivalently, adjoining the tested vector as a new column
leaves the corresponding rank unchanged.

These tests exhaust the monomials of degree \(d\). To justify
using one representative per orbit, let an ordered monomial
have an ordinary hit expression and apply a permutation to
that expression. Equivariance carries every source to a
source for the permuted monomial. Applying the inverse
permutation proves the converse. Ordinary hit membership is
therefore constant on each orbit, and all its ordered
monomials have the same orbit sum.

\begin{table}[H]
\centering
\caption{Exact ranks in the five initial degrees. The ordinary cohit
dimensions agree with Sum \cite[Theorem~4]{Sum2013}.}
\label{tab:finite-initial-ranks}
\begin{tabular}{r|rrrrrr|rr}
\toprule
\(d\)&\(N_d\)&\(b_d\)&\(r_d\)&\(s_d\)&\(a_d\)&\(c_d\)&
\(\dim Q_d(P(4))\)&\(\dim Q_d(B(4))\)\\
\midrule
1 & 4 & 1 & 0 & 0 & 1 & 0 & 4 & 1\\
3 & 20 & 3 & 6 & 1 & 8 & 0 & 14 & 2\\
6 & 84 & 9 & 60 & 7 & 62 & 0 & 24 & 2\\
10 & 286 & 23 & 216 & 19 & 220 & 0 & 70 & 4\\
16 & 969 & 64 & 896 & 58 & 901 & 1 & 73 & 6\\
\bottomrule
\end{tabular}
\end{table}

\subsection{The initial degrees one, three, and six}

Proposition~\ref{prop:eventual-kameko-stability} reduces the two
families to five initial degrees. The calculations in degrees one,
three, and six can be written completely in terms of individual
orbit sums.

\begin{proposition}\label{prop:small-initial-degrees}
Let \(d\in\{1,3,6\}\). Then the classes indexed by
\(\mathscr A_d\) form a basis of \(Q_d(B(4))\), and
\(\Crel_d(4)=0\). The ordinary cohit dimensions in these degrees
are \(4,14,24\), respectively \cite[Theorem~4]{Sum2013}.
\end{proposition}
\begin{proof}
In degree one there are no positive squares with a source of
sufficient degree. Thus the four ordinary monomials form an
ordinary cohit basis, the symmetric cohit basis consists of
\([m_{1,0,0,0}]\), and its ordinary image is nonzero.

In degree three the only square index is one. The two symmetric
source partitions are \((2,0,0,0)\) and \((1,1,0,0)\), and
\[
 \Sq^1m_{2,0,0,0}=0,\qquad
 \Sq^1m_{1,1,0,0}=m_{2,1,0,0}.
\]
The other target orbit sums are \(m_{1,1,1,0}\) and \(m_{3,0,0,0}\).
Their classes are detected by \(b_1^3\) and \(b_3\), respectively.
The spike coefficient functionals
\(v^{(1,1,1,0)}\) and \(v^{(3,0,0,0)}\) detect their ordinary
images independently. Indeed, a positive square has zero
coefficient on a spike by Lemma~\ref{lem:lucas} and the
Cartan formula~\eqref{eq:cartan}. Therefore the symmetric
cohit dimension is two and the relative space is zero.
The six ordinary columns
\(\Sq^1(x_i x_j)=x_i^2x_j+x_ix_j^2\), \(1\leq i<j\leq4\),
have disjoint nonempty supports; all other degree-two source
monomials have zero first square. Hence the ordinary hit
rank is six and the ordinary cohit dimension is \(20-6=14\).

In degree six, the complete list of symmetric hit basis
identities is
\begin{equation}\label{eq:degree-six-hit-equations}
\begin{aligned}
 m_{2,2,1,1}&=\Sq^2m_{1,1,1,1},\\
 m_{2,2,2,0}&=\Sq^1m_{2,2,1,0},\\
 m_{3,2,1,0}&=\Sq^1(m_{3,1,1,0}+m_{2,2,1,0})
                         +\Sq^2m_{2,1,1,0},\\
 m_{4,1,1,0}&=\Sq^2m_{2,1,1,0}+\Sq^1m_{2,2,1,0},\\
 m_{4,2,0,0}&=\Sq^2m_{2,2,0,0},\\
 m_{5,1,0,0}&=\Sq^2(m_{3,1,0,0}+m_{2,2,0,0}),\\
 m_{6,0,0,0}&=\Sq^1m_{5,0,0,0}.
\end{aligned}
\end{equation}
For example, the three source expansions needed for the third
and fourth lines are
\[
\begin{aligned}
 \Sq^1m_{3,1,1,0}&=m_{4,1,1,0}+m_{3,2,1,0},\\
 \Sq^1m_{2,2,1,0}&=m_{2,2,2,0},\\
 \Sq^2m_{2,1,1,0}&=m_{4,1,1,0}+m_{2,2,2,0}.
\end{aligned}
\]
The first, fifth, sixth, and seventh lines of
\eqref{eq:degree-six-hit-equations} follow by the same
coefficient formula~\eqref{eq:coefficient-test}. In the sixth,
\(\Sq^2m_{3,1,0,0}=m_{5,1,0,0}+m_{4,2,0,0}\),
which gives the displayed cancellation.
The nine degree-six orbit sums consist of these seven distinct
hit vectors and \(m_{3,1,1,1},m_{3,3,0,0}\).
The Steenrod-annihilated symmetric functionals \(b_3b_1^3,b_3^2\) pair
with the latter two as the identity matrix. The ordinary spike
functionals \(v^{(3,1,1,1)},v^{(3,3,0,0)}\) give the same
independent pairings after inclusion. Thus the seven displayed
vectors form the entire symmetric hit basis, the latter two
give the cohit basis, and the relative space is zero.

For the ordinary dimension in degree six, the \(\Sq^1\)- and
\(\Sq^2\)-source blocks contain \(56\) and \(35\) columns.
Their cumulative ranks, obtained from the coefficient formula
\eqref{eq:cartan}, are \(35\) and \(60\).
Appendix~\ref{app:finite-verifier} reconstructs these columns and
checks the final ranks by exact elimination. Thus
\(\dim Q_6(P(4))=84-60=24\), in agreement with
Sum \cite[Theorem~4]{Sum2013}.
\end{proof}

\subsection{An explicit calculation in degree ten}
\label{subsec:degree-ten}

We now carry out the preceding stages for \(d=10\), including
the formation of individual columns and the resulting
membership certificates. This initial degree determines the stable
part of the family \(e_t\).

\begin{proposition}\label{prop:degree-ten-basis}
Let \(n=4\). Then the classes indexed by \(\mathscr A_{10}\)
form a basis of \(Q_{10}(B(4))\),
\(\dim Q_{10}(P(4))=70\), as in \cite[Theorem~4]{Sum2013},
and \(\Crel_{10}(4)=0\).
\end{proposition}
\begin{proof}

\par\medskip\noindent\textit{Bases and square indices.}
Enumeration of the target bases gives
\[
 N_{10}=\binom{13}{3}=286,\qquad b_{10}=23.
\]
Lemma~\ref{lem:power-generators} and instability restrict the
required square indices to \(1,2,4\). The corresponding
source counts, with zero columns retained, are
\[
 \begin{array}{c|c|r|r}
 r&10-r&N_{10-r}&b_{10-r}\\ \hline
 1&9&220&18\\
 2&8&165&15\\
 4&6&84&9\\ \hline
 \multicolumn{2}{c|}{\text{Total}}&469&42
 \end{array}
\]
Thus \(O_{10}\) has size \(286\times469\), \(S_{10}\) has
size \(23\times42\), and \(A_{10}\) has size \(286\times23\).
The augmented matrix has size \(286\times492\).
Exactly \(10\) ordinary columns and \(5\) symmetric columns
are zero, so deleting zero columns leaves \(459\) and \(37\)
columns, respectively.

\par\medskip\noindent\textit{Explicit ordinary and symmetric columns.}
Consider the source \(x_1^3x_2^3\) for \(\Sq^4\).
Lemma~\ref{lem:lucas} permits increments \(0,1,2,3\) in
each of its first two exponents. The only pairs with total
four are \((1,3)\), \((2,2)\), and \((3,1)\). The last two
exponents are zero and admit only zero increments. Hence
\begin{equation}\label{eq:degree-ten-ordinary-column}
 \Sq^4(x_1^3x_2^3)
 =x_1^4x_2^6+x_1^5x_2^5+x_1^6x_2^4.
\end{equation}
This column of \(O_{10}\) has a one in each of the three
displayed monomial rows and zero elsewhere.

For the corresponding symmetric source, applying
\eqref{eq:coefficient-test} gives
\begin{equation}\label{eq:degree-ten-symmetric-columns}
 \begin{aligned}
 \Sq^4m_{3,3,0,0}&=m_{6,4,0,0}+m_{5,5,0,0},\\
 \Sq^1m_{6,3,0,0}&=m_{6,4,0,0}.
 \end{aligned}
\end{equation}
For the first identity, each choice of two variable positions
produces one monomial of type \((5,5,0,0)\) and the two
ordered monomials of type \((6,4,0,0)\) supported on those
positions. Each target therefore occurs once. For the
second identity, only the exponent \(3\) is odd, and its
increment gives exponent \(4\).
The two target orbits have sizes \(12\) and \(6\).
Their coefficients are nevertheless both one in the first
identity, illustrating why a symmetric matrix entry must
be read at a single target representative.

\par\medskip\noindent\textit{Successive rank calculations.}
Adjoining the square blocks in the order \(1,2,4\) gives
the following cumulative ranks:
\[
 \begin{array}{c|r|r}
 \text{Square indices included}
 &\text{Ordinary rank}&\text{Symmetric rank}\\ \hline
 \{1\}&126&11\\
 \{1,2\}&216&17\\
 \{1,2,4\}&216&19
 \end{array}
\]
These are ranks of the combined blocks. In particular,
the \(\Sq^4\)-columns add two independent symmetric
directions after the \(\Sq^1\)- and \(\Sq^2\)-columns.
Adjoining the \(23\) columns of \(A_{10}\) to \(O_{10}\)
raises the rank from \(216\) to \(220\). Therefore
\[
 r_{10}=216,\qquad s_{10}=19,\qquad a_{10}=220.
\]
Proposition~\ref{prop:rank-formula} now gives
\[
 \begin{aligned}
 \dim Q_{10}(P(4))&=286-216=70,\\
 \dim Q_{10}(B(4))&=23-19=4,\\
 \rank\bar\iota_{10}&=220-216=4,\\
 c_{10}&=23-19-220+216=0.
 \end{aligned}
\]
This obtains the degree-ten row of
Table~\ref{tab:finite-initial-ranks} directly from the three
matrix reductions.

\par\medskip\noindent\textit{Representatives in symmetric cohits.}
A basis obtained from the reduction is
\begin{equation}\label{eq:degree-ten-cohit-basis}
 [m_{3,3,2,2}],\qquad [m_{3,3,3,1}],\qquad
 [m_{7,1,1,1}],\qquad [m_{7,3,0,0}].
\end{equation}
Their independence in symmetric cohits can also be
checked by the four symmetric dual functionals
\[
 b_3^2b_2^2+b_5b_3b_1^2,\qquad
 b_3^3b_1,\qquad b_7b_1^3,\qquad b_7b_3.
\]
For the first functional, \eqref{eq:dual-cartan} gives
zero under \(\Sq_*^1\) and \(\Sq_*^4\), while its
\(\Sq_*^2\)-image is
\[
 b_3^2b_1^2+b_3^2b_1^2=0.
\]
Each factor of the other three functionals is a spike
Steenrod-annihilated, as follows from Lemma~\ref{lem:lucas} and
\eqref{eq:dual-cartan}; the Cartan formula therefore
annihilates each product under every positive square.
Larger power-of-two down squares vanish in degree ten by
instability of their adjoints. Lemma~\ref{lem:power-generators}
shows that all four functionals belong to
\(K_{10}^\Sigma(4)\). Their pairings, in the displayed order,
with \eqref{eq:degree-ten-cohit-basis} form the identity
matrix. A hit linear combination of these representatives
would consequently have all four coefficients zero.
Since the symmetric cohit dimension is four, the
representatives form a basis. Their ordinary images are
also independent because \(\bar\iota_{10}\) has rank four.

The same reduction gives, for example, the exact relation
\begin{equation}\label{eq:degree-ten-cohit-relation}
 \begin{aligned}
 m_{5,3,1,1}+m_{3,3,2,2}
 ={}&\Sq^1(m_{3,2,2,2}+m_{3,3,2,1})\\
 &+\Sq^2m_{3,3,1,1}+\Sq^4m_{2,2,1,1}.
 \end{aligned}
\end{equation}
To verify it by \eqref{eq:coefficient-test}, the required
expansions are
\[
 \begin{aligned}
 \Sq^1m_{3,2,2,2}&=m_{4,2,2,2},\\
 \Sq^1m_{3,3,2,1}&=m_{4,3,2,1},\\
 \Sq^2m_{3,3,1,1}
 &=m_{5,3,1,1}+m_{4,4,1,1}
   +m_{4,3,2,1}+m_{3,3,2,2},\\
 \Sq^4m_{2,2,1,1}&=m_{4,4,1,1}+m_{4,2,2,2}.
 \end{aligned}
\]
Adding the four lines cancels every term except the two
on the left side of \eqref{eq:degree-ten-cohit-relation}.

\par\medskip\noindent\textit{Monomial membership and explicit preimages.}
The ordinary membership test succeeds for exactly the
following eleven partition representatives:
\[
 \begin{gathered}
 (10,0,0,0),\ (9,1,0,0),\ (8,2,0,0),\
 (8,1,1,0),\\
 (6,4,0,0),\ (6,2,2,0),\ (5,5,0,0),\
 (5,4,1,0),\\
 (4,4,2,0),\ (4,4,1,1),\ (4,2,2,2).
 \end{gathered}
\]
Their orbits contain \(116\) ordered monomials in total.
The symmetric membership test succeeds for every one of
the eleven corresponding orbit sums. The number \(116\)
counts monomials lying in the ordinary hit space, whereas
\(216\) is the dimension of that space.

For an explicit example with repeated exponents,
\eqref{eq:degree-ten-ordinary-column} and the first-square
expansions give
\[
 \begin{aligned}
 \Sq^1(x_1^3x_2^6+x_1^6x_2^3)
   &=x_1^4x_2^6+x_1^6x_2^4,\\
 x_1^5x_2^5
   &=\Sq^4(x_1^3x_2^3)
     +\Sq^1(x_1^3x_2^6+x_1^6x_2^3).
 \end{aligned}
\]
Both ordinary preimages are fixed by
\(\langle(12),(34)\rangle\). Taking their coset norms,
or adding the two identities in
\eqref{eq:degree-ten-symmetric-columns}, gives
\[
 m_{5,5,0,0}
 =\Sq^4m_{3,3,0,0}+\Sq^1m_{6,3,0,0}.
\]
Thus the ordinary and symmetric membership assertions
are both certified by explicit Cartan identities.
The passage between them is an instance of the
stabilizer calculation in Lemma~\ref{lem:parity}.

\end{proof}

\subsection{The degree-sixteen basis and relative class}

Degree sixteen is the last initial degree in the family \(d_t\).
Here a symmetric ordinary hit survives in symmetric cohits. We
first certify a complete quotient basis and then determine its
ordinary image by explicit dual functionals and a Cartan identity.

\begin{lemma}\label{lem:degree-sixteen-dual-basis}
Let \(1\leq i\leq6\). Then \(\Psi_i\in K_{16}^\Sigma(4)\).
The evaluation matrix of \(\Psi_1,\ldots,\Psi_6\) on the ordered
list \(\mathscr A_{16}\) is the identity matrix.
\end{lemma}
\begin{proof}
Theorem~\ref{thm:fifteen} gives
$\Theta_{15}\in K_{15}^\Sigma(3)$. The factor $b_1$ is Steenrod-annihilated.
Equation~\eqref{eq:dual-cartan} therefore gives
$\Psi_1\in K_{16}^\Sigma(4)$. The functionals
$\Psi_2,\Psi_4,\Psi_6$ are products of the Steenrod-annihilated spike
elements $b_1,b_3,b_7,b_{15}$, so the same Cartan formula
annihilates them under every positive down square.

For $\Psi_3$, the only generator indices that can act nontrivially
in degree sixteen are $1,2,4,8$, by
Lemma~\ref{lem:power-generators} and adjoint instability. Their
values on its four summands are
\[
\begin{array}{c|cccc}
 &\Sq_*^1&\Sq_*^2&\Sq_*^4&\Sq_*^8\\\hline
 b_7b_4^2b_1&0&b_7b_3^2b_1&b_7b_2^2b_1&0\\
 b_7b_5b_3b_1&0&b_7b_3^2b_1&0&0\\
 b_{11}b_2^2b_1&0&b_{11}b_1^3&b_7b_2^2b_1&0\\
 b_{13}b_1^3&0&b_{11}b_1^3&0&0
\end{array}
\]
For example, $\Sq_*^2(b_4^2)=(\Sq_*^1b_4)^2=b_3^2$,
$\Sq_*^4(b_4^2)=(\Sq_*^2b_4)^2=b_2^2$, and
$\Sq_*^2b_{13}=\binom{11}{2}b_{11}=b_{11}$.
All remaining entries follow from the same binomial formula in
\eqref{eq:dual-cartan}. Every column sums to zero, proving the
required vanishing for $\Psi_3$.

The eight summands of $\Psi_5$ give the following table:
\[
\begin{array}{c|cccc}
 &\Sq_*^1&\Sq_*^2&\Sq_*^4&\Sq_*^8\\\hline
 b_7^2b_2&b_7^2b_1&0&0&0\\
 b_8b_7b_1&b_7^2b_1&b_7b_6b_1&b_7b_4b_1&0\\
 b_9b_6b_1&b_9b_5b_1&b_7b_6b_1&b_6b_5b_1&0\\
 b_{10}b_5b_1&b_9b_5b_1&b_{10}b_3b_1&b_6b_5b_1&0\\
 b_{11}b_4b_1&b_{11}b_3b_1&b_{11}b_2b_1&b_7b_4b_1&0\\
 b_{12}b_3b_1&b_{11}b_3b_1&b_{10}b_3b_1&0&0\\
 b_{13}b_2b_1&b_{13}b_1^2&b_{11}b_2b_1&0&0\\
 b_{14}b_1^2&b_{13}b_1^2&0&0&0
\end{array}
\]
Each entry is obtained from \eqref{eq:dual-cartan}, and again each
column sums to zero. This proves that $\Psi_5$ is in the symmetric
Steenrod kernel. Finally the coefficient of the $i$th partition in
$\Psi_i$ is one, and no other member of $\mathscr A_{16}$ occurs in
$\Psi_i$. The definition of the orbit-dual basis gives the asserted
identity evaluation matrix.
\end{proof}

\begin{proposition}\label{prop:degree-sixteen-basis}
Let \(n=4\). Then
\(\dim Q_{16}(P(4))=73\), as in \cite[Theorem~4]{Sum2013},
\(\dim\Hit^{16}(B(4))=58\), and
\begin{equation}\label{eq:degree-sixteen-cohit-basis}
\begin{gathered}
 [m_{5,5,5,1}],\quad [m_{7,3,3,3}],\quad [m_{7,4,4,1}],\\
 [m_{7,7,1,1}],\quad [m_{7,7,2,0}],\quad [m_{15,1,0,0}]
\end{gathered}
\end{equation}
form a basis of \(Q_{16}(B(4))\).
\end{proposition}
\begin{proof}
The target dimensions are \(N_{16}=\binom{19}{3}=969\)
and \(b_{16}=64\). The symmetric matrix is formed from the
source orbit sums in degrees \(15,14,12,8\), with square
indices \(1,2,4,8\), respectively.
\par\medskip\noindent\textit{Ordinary and symmetric ranks in degree sixteen.}
The source counts and cumulative ranks are
\begin{equation}\label{eq:degree-sixteen-rank-contributions}
\begin{array}{c|r|r|r|r|r}
 r&N_{16-r}&b_{16-r}&
 \text{Nonzero symmetric columns}&
 \text{Ordinary rank}&\text{Symmetric rank}\\\hline
 1&816&54&45&445&29\\
 2&680&47&43&760&47\\
 4&455&34&33&896&57\\
 8&165&15&15&896&58
\end{array}
\end{equation}
The last two columns record the ranks after adjoining all preceding
square blocks. All \(2116\) ordinary columns are nonzero.
The symmetric matrix has \(150\) columns, of which \(136\) are
nonzero. Its rank increments are \(29,18,10,1\).
The \(58\) distinct elimination pivots establish the spanning
calculation, while the six symmetric kernel functionals establish
the complementary lower bound on the cohit dimension.

The selected columns in Table~\ref{tab:sixteen-pivots} supply
an explicit rank lower bound. For each listed source, expand
\(\Sq^r m_a\) using \eqref{eq:coefficient-test}; eliminate
the previously stored largest target indices, with targets
ordered increasingly in lexicographic order. The remaining
nonzero column has the displayed pivot. These \(58\) pivots
are distinct and are exactly the partitions outside
\(\mathscr A_{16}\). The executable calculation in Appendix~\ref{app:finite-verifier}
reconstructs these Cartan columns and performs exactly this
largest-coordinate elimination. A nonzero remainder with a new pivot
cannot belong to the span of the preceding columns: every nonzero
linear combination of the stored columns has largest coordinate
at one of their distinct pivots. Conversely, subtracting the stored
columns preserves the class modulo their span. Induction over the
listed sources therefore proves independence of all \(58\) columns.
Thus the \(58\) original columns
are independent and \(s_{16}\geq58\).
For example, \(\Sq^1m_{5,4,3,3}=m_{6,4,3,3}\),
while
\[
 \Sq^1m_{5,5,3,2}=m_{6,5,3,2}+m_{5,5,4,2},
 \qquad \Sq^8m_{2,2,2,2}=m_{4,4,4,4}.
\]
The last equality is the instability identity
\(\Sq^{|f|}f=f^2\) for a degree-eight monomial.
These give the first, third, and last pivot entries.

\begin{table}[H]
\centering\small
\caption{Fifty-eight independent symmetric Cartan columns in degree sixteen.
Read the left block from top to bottom, then the right block.}
\label{tab:sixteen-pivots}
\begin{tabular}{rcc|rcc}
\toprule
\(r\)&Source&Elimination pivot&\(r\)&Source&Elimination pivot\\
\midrule
1&$(5,4,3,3)$&$(6,4,3,3)$&2&$(4,4,3,3)$&$(5,4,4,3)$\\
1&$(5,4,4,2)$&$(6,4,4,2)$&2&$(5,5,2,2)$&$(6,6,2,2)$\\
1&$(5,5,3,2)$&$(6,5,3,2)$&2&$(5,5,3,1)$&$(6,6,3,1)$\\
1&$(5,5,4,1)$&$(6,5,4,1)$&2&$(6,3,3,2)$&$(7,4,3,2)$\\
1&$(5,5,5,0)$&$(6,5,5,0)$&2&$(6,5,2,1)$&$(7,6,2,1)$\\
1&$(6,6,3,0)$&$(6,6,4,0)$&2&$(6,5,3,0)$&$(7,6,3,0)$\\
1&$(7,3,3,2)$&$(8,3,3,2)$&2&$(7,3,2,2)$&$(9,3,2,2)$\\
1&$(7,4,2,2)$&$(8,4,2,2)$&2&$(7,3,3,1)$&$(9,3,3,1)$\\
1&$(7,4,3,1)$&$(8,4,3,1)$&2&$(7,4,2,1)$&$(9,4,2,1)$\\
1&$(7,4,4,0)$&$(8,4,4,0)$&2&$(7,4,3,0)$&$(9,4,3,0)$\\
1&$(7,5,2,1)$&$(8,5,2,1)$&2&$(7,5,1,1)$&$(9,5,1,1)$\\
1&$(7,5,3,0)$&$(8,5,3,0)$&2&$(7,5,2,0)$&$(9,5,2,0)$\\
1&$(7,6,1,1)$&$(8,6,1,1)$&2&$(7,6,1,0)$&$(9,6,1,0)$\\
1&$(7,6,2,0)$&$(8,6,2,0)$&2&$(7,7,0,0)$&$(9,7,0,0)$\\
1&$(7,7,1,0)$&$(8,7,1,0)$&2&$(10,3,1,0)$&$(11,4,1,0)$\\
1&$(9,2,2,2)$&$(10,2,2,2)$&2&$(11,1,1,1)$&$(13,1,1,1)$\\
1&$(9,3,2,1)$&$(10,3,2,1)$&2&$(11,2,1,0)$&$(13,2,1,0)$\\
1&$(9,3,3,0)$&$(10,3,3,0)$&2&$(11,3,0,0)$&$(13,3,0,0)$\\
1&$(9,4,1,1)$&$(10,4,1,1)$&4&$(3,3,3,3)$&$(5,5,3,3)$\\
1&$(9,4,2,0)$&$(10,4,2,0)$&4&$(4,4,3,1)$&$(5,5,4,2)$\\
1&$(9,5,1,0)$&$(10,5,1,0)$&4&$(5,3,2,2)$&$(7,5,2,2)$\\
1&$(9,6,0,0)$&$(10,6,0,0)$&4&$(5,3,3,1)$&$(7,5,3,1)$\\
1&$(11,2,1,1)$&$(12,2,1,1)$&4&$(5,4,3,0)$&$(7,5,4,0)$\\
1&$(11,2,2,0)$&$(12,2,2,0)$&4&$(6,6,0,0)$&$(8,8,0,0)$\\
1&$(11,3,1,0)$&$(12,3,1,0)$&4&$(7,2,2,1)$&$(11,2,2,1)$\\
1&$(11,4,0,0)$&$(12,4,0,0)$&4&$(7,3,1,1)$&$(11,3,1,1)$\\
1&$(13,1,1,0)$&$(14,1,1,0)$&4&$(7,3,2,0)$&$(11,3,2,0)$\\
1&$(13,2,0,0)$&$(14,2,0,0)$&4&$(7,5,0,0)$&$(11,5,0,0)$\\
1&$(15,0,0,0)$&$(16,0,0,0)$&8&$(2,2,2,2)$&$(4,4,4,4)$\\
\bottomrule
\end{tabular}
\end{table}

Lemma~\ref{lem:degree-sixteen-dual-basis} gives six independent
functionals annihilating the entire symmetric hit space.
Since \(B^{16}(4)\) has dimension \(64\), it follows that
\(s_{16}\leq64-6=58\). Combining the two bounds gives
\(s_{16}=58\), and the identity pairing from that lemma
makes \eqref{eq:degree-sixteen-cohit-basis} a quotient basis.
The ordinary matrix reductions in
\eqref{eq:degree-sixteen-rank-contributions} give
\(r_{16}=896\). The complete matrix and its rank are reproduced in
Appendix~\ref{app:finite-verifier}. Hence
\(\dim Q_{16}(P(4))=969-896=73\), consistent with
Sum \cite[Theorem~4]{Sum2013}.
\end{proof}

\begin{proposition}\label{prop:degree-sixteen-relative}
Let \(n=4\). Then
\[
 \operatorname{rank}\bar\iota_{16}=5,
 \qquad
 \Crel_{16}(4)=\F\,[m_{7,7,2,0}].
\]
Every ordered monomial occurring in \(m_{7,7,2,0}\) is
ordinary non-hit.
\end{proposition}
\begin{proof}
We first verify that the five functionals in
\eqref{eq:degree-sixteen-ordinary-detectors} belong to
\(K_{16}(4)\). The ordinary divided-power computations use
\eqref{eq:ordinary-dual-cartan}.
In \(U\), only the three terms with two exponents equal to six
have a nonzero first down square. Their images are
\[
\begin{aligned}
 \Sq_*^1v^{(3,6,6)}&=v^{(3,5,6)}+v^{(3,6,5)},\\
 \Sq_*^1v^{(6,3,6)}&=v^{(5,3,6)}+v^{(6,3,5)},\\
 \Sq_*^1v^{(6,6,3)}&=v^{(5,6,3)}+v^{(6,5,3)}.
\end{aligned}
\]
The other seven terms have only odd exponents; their first
down squares vanish because \(\binom{j-1}{1}\) is even for odd
\(j\). Therefore \(\Sq_*^1U=V\).
The remaining power-of-two down squares on \(U\) are given by
the following complete table.
\begingroup\small
\[
\begin{array}{c|cc}
a&\Sq_*^2v^{(a)}&\Sq_*^4v^{(a)}\\\hline
(3,3,9)&v^{(3,3,7)}&v^{(3,3,5)}\\
(3,6,6)&v^{(3,5,5)}&v^{(3,3,5)}+v^{(3,5,3)}\\
(3,9,3)&v^{(3,7,3)}&v^{(3,5,3)}\\
(5,3,7)&v^{(3,3,7)}&0\\
(5,5,5)&v^{(3,5,5)}+v^{(5,3,5)}+v^{(5,5,3)}&v^{(3,3,5)}+v^{(3,5,3)}+v^{(5,3,3)}\\
(5,7,3)&v^{(3,7,3)}&0\\
(6,3,6)&v^{(5,3,5)}&v^{(3,3,5)}+v^{(5,3,3)}\\
(6,6,3)&v^{(5,5,3)}&v^{(3,5,3)}+v^{(5,3,3)}\\
(7,5,3)&v^{(7,3,3)}&0\\
(9,3,3)&v^{(7,3,3)}&v^{(5,3,3)}\\
\end{array}
\]
\endgroup
Every divided-power term in either image column occurs an even
number of times. Thus \(\Sq_*^2U=\Sq_*^4U=0\).
For \(V\), the corresponding table is
\begingroup\small
\[
\begin{array}{c|ccc}
a&\Sq_*^1v^{(a)}&\Sq_*^2v^{(a)}&\Sq_*^4v^{(a)}\\\hline
(3,5,6)&v^{(3,5,5)}&v^{(3,3,6)}&0\\
(3,6,5)&v^{(3,5,5)}&v^{(3,6,3)}&0\\
(5,3,6)&v^{(5,3,5)}&v^{(3,3,6)}&0\\
(5,6,3)&v^{(5,5,3)}&v^{(3,6,3)}&0\\
(6,3,5)&v^{(5,3,5)}&v^{(6,3,3)}&0\\
(6,5,3)&v^{(5,5,3)}&v^{(6,3,3)}&0\\
\end{array}
\]
\endgroup
The first two columns again sum to zero, and the fourth-square
column is zero term by term. Indices at least eight vanish on
\(\DP_{14}(3)\) by adjoint instability. Hence
Lemma~\ref{lem:power-generators} gives
\(V\in K_{14}(3)\).

On the first tensor factor,
\(\Sq_*^1v^{(1)}=0\),
\(\Sq_*^1v^{(2)}=v^{(1)}\), and every down square of
index at least two vanishes on both \(v^{(1)}\) and
\(v^{(2)}\). Consequently
\[
\begin{aligned}
 \Sq_*^1P_{16}
 &=v^{(1)}\otimes\Sq_*^1U
   +v^{(1)}\otimes V+v^{(2)}\otimes\Sq_*^1V\\
 &=v^{(1)}\otimes V+v^{(1)}\otimes V=0.
\end{aligned}
\]
For \(r\in\{2,4,8\}\), the only possible tensor contributions are
\[
 \Sq_*^rP_{16}
 =v^{(1)}\otimes\Sq_*^rU
  +v^{(1)}\otimes\Sq_*^{r-1}V
  +v^{(2)}\otimes\Sq_*^rV.
\]
The last two summands vanish because \(V\) is Steenrod-annihilated.
The first summand vanishes for \(r=2,4\) by the displayed
table and for \(r=8\) by adjoint instability in degree fifteen.
All larger generator indices vanish in degree sixteen.
Lemma~\ref{lem:power-generators} now proves
\(P_{16}\in K_{16}(4)\).

Lemma~\ref{lem:naturality-spikes} gives
\(\Pi_2,\Pi_4,\Pi_5\in K_{16}(4)\) and
\(v^{(15,1)}\in K_{16}(2)\).

By Lemma~\ref{lem:naturality-spikes}, \(\pi_{16}\) commutes
with every square. The coefficient of \(y_1^{15}y_2\) in
\(\pi_{16}(x_1^ax_2^bx_3^cx_4^d)\) is one precisely when
\(a+b+c=15\) and \(d=1\). Therefore
\(F_{16}=\pi_{16}^*(v^{(15,1)})\).
Taking transposes of the commuting squares gives
\(\Sq_*^rF_{16}=\pi_{16}^*\Sq_*^r(v^{(15,1)})=0\)
for every \(r>0\). This establishes the claimed annihilation
of all five functionals.

Their pairings with the six representatives in
\eqref{eq:degree-sixteen-cohit-basis} are
\begin{equation}\label{eq:degree-sixteen-ordinary-pairing}
\begin{array}{c|rrrrrr}
 &m_{5,5,5,1}&m_{7,3,3,3}&m_{7,4,4,1}&
 m_{7,7,1,1}&m_{7,7,2,0}&m_{15,1,0,0}\\\hline
 \Pi_1&1&0&0&0&0&0\\
 \Pi_2&0&1&0&0&0&0\\
 \Pi_3&1&0&1&1&0&1\\
 \Pi_4&0&0&0&1&0&0\\
 \Pi_5&0&0&0&0&0&1
\end{array}
\end{equation}
For \(P_{16}\), the only ordered monomial shared by its support
and these six orbit sums is the tuple \((1,5,5,5)\) in the
first orbit. Each of the three spike functionals is supported
on one ordered tuple in its indicated orbit and none in the
other five. For \(F_{16}\), the numbers of terms with fourth
exponent one in the six orbits are
\(1,0,3,3,0,3\), respectively. Reducing these numbers modulo two
gives its row of \eqref{eq:degree-sixteen-ordinary-pairing}.
The submatrix on columns \(1,2,3,4,6\) has determinant one:
adding rows \(1,4,5\) to row \(3\) makes it the identity matrix.
Thus the ordinary images of those five classes are independent.

It remains to certify the ordinary hit identity for the fifth
representative. The three source polynomials are those of
\eqref{eq:degree-sixteen-ordinary-sources}.
The following table gives every nonzero coefficient arising in
the three Cartan expansions. Its first column records the exponent
triple of the ordered monomial; an omitted triple has coefficient
zero in all three columns.
\begin{longtable}{c|ccc}
\caption{Complete ordered Cartan coefficients in degree sixteen.}
\label{tab:degree-sixteen-cartan}\\
(a,b,c)&$\Sq^1h_1^{(3)}$&$\Sq^2h_2^{(3)}$&$\Sq^4h_4^{(3)}$\\\hline
\endfirsthead
(a,b,c)&$\Sq^1h_1^{(3)}$&$\Sq^2h_2^{(3)}$&$\Sq^4h_4^{(3)}$\\\hline
\endhead
$(1,5,10)$&1&0&1\\
$(1,6,9)$&1&1&0\\
$(1,7,8)$&1&1&0\\
$(1,8,7)$&1&1&0\\
$(1,9,6)$&0&1&1\\
$(2,4,10)$&1&1&0\\
$(2,5,9)$&1&1&0\\
$(2,6,8)$&0&1&1\\
$(2,7,7)$&1&0&0\\
$(2,8,6)$&1&1&0\\
$(3,3,10)$&1&0&1\\
$(3,4,9)$&1&1&0\\
$(3,9,4)$&0&1&1\\
$(4,2,10)$&1&1&0\\
$(4,3,9)$&1&1&0\\
$(4,4,8)$&0&1&1\\
$(4,6,6)$&0&1&1\\
$(4,10,2)$&1&1&0\\
$(5,1,10)$&1&0&1\\
$(5,2,9)$&1&1&0\\
$(5,3,8)$&0&1&1\\
$(5,5,6)$&0&1&1\\
$(5,9,2)$&1&1&0\\
$(5,10,1)$&1&0&1\\
$(6,1,9)$&1&1&0\\
$(6,2,8)$&0&1&1\\
$(6,4,6)$&1&0&1\\
$(6,6,4)$&1&0&1\\
$(6,8,2)$&0&1&1\\
$(6,9,1)$&1&1&0\\
$(7,1,8)$&1&1&0\\
$(7,2,7)$&1&0&0\\
$(7,7,2)$&1&0&0\\
$(7,8,1)$&1&1&0\\
$(8,1,7)$&1&1&0\\
$(8,2,6)$&1&1&0\\
$(8,6,2)$&1&1&0\\
$(8,7,1)$&1&1&0\\
$(9,1,6)$&0&1&1\\
$(9,6,1)$&0&1&1\\
\hline
\end{longtable}
Each entry is obtained by substituting the relevant source tuple
from \(\mathcal D_r\) into \eqref{eq:cartan} and reducing the
product of integer binomial coefficients modulo two.
For example, the coefficient at \((2,7,7)\) in the first column
comes from the source \((1,7,7)\) and the increment \((1,0,0)\),
whose coefficient is one. In the second column, the only coordinatewise
smaller source triples are \((1,6,7)\) and \((1,7,6)\); each
contribution contains the even factor \(\binom61\).
In the fourth-square column, the only coordinatewise smaller source
is \((1,5,6)\), and its increment \((1,2,1)\) again contains
\(\binom61\). This gives the row \((1,0,0)\).
The same finite coefficient formula gives all rows of the table.
After summing the three columns, only \((2,7,7)\),
\((7,2,7)\), and \((7,7,2)\) remain. Hence
\begin{equation}\label{eq:degree-sixteen-three-variable-hit}
 m_{7,7,2}(x,y,z)
 =\Sq^1h_1^{(3)}+\Sq^2h_2^{(3)}+\Sq^4h_4^{(3)}.
\end{equation}
Each \(\eta_j\) commutes with squares by
Lemma~\ref{lem:naturality-spikes}.
Every monomial in \(m_{7,7,2,0}\) has exactly three positive
exponents, so it occurs in precisely one of the four support
embeddings of \(m_{7,7,2}\). Applying the four embeddings to
\eqref{eq:degree-sixteen-three-variable-hit} and summing gives
\begin{equation}\label{eq:degree-sixteen-four-variable-hit}
 m_{7,7,2,0}
 =\sum_{r\in\{1,2,4\}}\Sq^r
       \left(\sum_{j=1}^4\eta_j(h_r^{(3)})\right).
\end{equation}
This identity puts the fifth representative in the ordinary hit
space. It is nonzero in symmetric cohits by
\eqref{eq:degree-sixteen-cohit-basis}.
Together with the five independent ordinary images established
by \eqref{eq:degree-sixteen-ordinary-pairing}, this proves both
the inclusion rank and the stated description of
\(\Crel_{16}(4)\).

Every constituent monomial of \(m_{7,7,2,0}\) has an odd
exponent and degree \(16\). Proposition~\ref{prop:power-of-two-monomials}
therefore makes each constituent ordinary non-hit.
\end{proof}

\subsection{The monomial implication in the finite initial degrees}

The relative class just exhibited is a sum of individually non-hit
monomials. The monomial implication therefore requires a separate
argument before the degree isomorphisms can be applied.

\begin{proposition}\label{prop:finite-monomial-implication}
Let \(d\in\{1,3,6,10,16\}\), and let \(x^a\) be a monomial
of degree \(d\). If \(x^a\in\Hit(P(4))\), then
\(m_a\in\Hit(B(4))\).
\end{proposition}
\begin{proof}
For \(d=1,3,6,10\), the relative space is zero by
Propositions~\ref{prop:small-initial-degrees} and
\ref{prop:degree-ten-basis}. Corollary~\ref{cor:relative-detection} therefore gives the result.
For degree sixteen, Proposition~\ref{prop:power-of-two-monomials}
gives a direct classification: an ordered monomial is hit precisely
when all four exponents are even. For every such tuple,
\[
 x^a=\Sq^8(x^{a/2}),\qquad m_a=\Sq^8(m_{a/2}).
\]
Thus Janfada's implication \cite[Conjecture~1]{Janfada2011} holds
in degree sixteen without a
separate membership calculation. Halving identifies the hit
monomials with all degree-eight monomials in four variables,
so their number is \(\binom{11}{3}=165\). Likewise, their
orbit representatives are exactly twice the \(15\) partitions
of eight of length at most four. This recovers the finite counts
and supplies a single explicit preimage formula for every case.
\end{proof}

\subsection{Complete bases in two generic families}

The preceding finite calculations now determine every member
of the two families. The distinction between a finite initial
calculation and a transported basis is essential: the subsequent
dimensions follow from isomorphisms, and their bases are obtained
by explicit exponent formulas.

\begin{theorem}\label{thm:five-generic-family}
Let \(t\geq0\). Then
\[
\begin{array}{c|ccc}
 t&\dim Q_{d_t}(P(4))&\dim Q_{d_t}(B(4))&\dim\Crel_{d_t}(4)\\ \hline
 0&4&1&0\\
 1&24&2&0\\
 t\geq2&73&6&1 .
\end{array}
\]
If \(t\geq2\), then
\begin{equation}\label{eq:five-family-cohit-basis}
 \bigl\{[m_{a^{[t-2]}}]:a\in\mathscr A_{16}\bigr\}
\end{equation}
is a basis of \(Q_{d_t}(B(4))\). Its fifth member spans
\(\Crel_{d_t}(4)\), on which the normalizer quotient acts
trivially. Every ordered monomial in that fifth orbit is
ordinary non-hit. If a monomial of degree \(d_t\) is
ordinary hit, then its orbit sum is symmetrically hit.
\end{theorem}
\begin{proof}
Proposition~\ref{prop:small-initial-degrees} gives the dimensions
at \(t=0,1\), corresponding to degrees one and six.
Propositions~\ref{prop:degree-sixteen-basis} and
\ref{prop:degree-sixteen-relative} give the dimensions,
cohit basis, and relative generator at \(t=2\).
For \(e=1\), the digit sums in
Lemma~\ref{lem:kameko-binary-criterion} are
\[
 \alpha(1+4)=2,\qquad \alpha(1+3)+t-1=t .
\]
Consequently \(\mu(d_t)=4\) for \(t\geq1\) precisely when
\(t\geq3\). Equivalently, \(s_4(1)=2\).
Propositions~\ref{prop:kameko-cohit},
\ref{prop:kameko-relative}, and
\ref{prop:eventual-kameko-stability} identify all three
spaces at index \(t\geq2\) with their spaces at index two.
This proves every asserted dimension. The ordinary dimensions
agree with Sum \cite[Theorem~4]{Sum2013}, followed by the
ordinary Kameko isomorphisms.
Proposition~\ref{prop:transport-complete-bases} gives
\eqref{eq:five-family-cohit-basis}; the relative
isomorphism transports its fifth member.
Every automorphism of a one-dimensional vector space over
\(\F\) is the identity, proving the assertion about the action.

The degree-sixteen constituent monomials are non-hit by
Proposition~\ref{prop:degree-sixteen-relative}.
Since \(\alpha(20)=2\) and \(\alpha(19)=3\),
Theorem~\ref{thm:binary-classification}, with initial degree
sixteen, preserves this non-hit property at every higher index.
For an arbitrary individually hit monomial at an index \(t\geq2\),
that theorem either makes its orbit sum hit through a failed
binary congruence or reduces both membership questions to degree
sixteen. Proposition~\ref{prop:finite-monomial-implication}
settles the latter case. In degrees one and six, the relative
space is zero by Proposition~\ref{prop:small-initial-degrees}.
Corollary~\ref{cor:relative-detection} gives the monomial
implication in these two initial degrees.
\end{proof}

The six exponent tuples in
\eqref{eq:five-family-cohit-basis} are explicitly
\begin{equation}\label{eq:five-family-affine-exponents}
\begin{aligned}
 &(6\cdot2^{t-2}-1,\ 6\cdot2^{t-2}-1,\
                   6\cdot2^{t-2}-1,\ 2^{t-1}-1),\\
 &(2^{t+1}-1,\ 2^t-1,\ 2^t-1,\ 2^t-1),\\
 &(2^{t+1}-1,\ 5\cdot2^{t-2}-1,\
                   5\cdot2^{t-2}-1,\ 2^{t-1}-1),\\
 &(2^{t+1}-1,\ 2^{t+1}-1,\ 2^{t-1}-1,\ 2^{t-1}-1),\\
 &(2^{t+1}-1,\ 2^{t+1}-1,\
                   3\cdot2^{t-2}-1,\ 2^{t-2}-1),\\
 &(2^{t+2}-1,\ 2^{t-1}-1,\ 2^{t-2}-1,\ 2^{t-2}-1).
\end{aligned}
\end{equation}
Each tuple has sum \(d_t\). The fifth tuple therefore gives
an explicit infinite family of symmetric ordinary hits whose
ordered constituents are non-hit and whose symmetric cohit
classes are nonzero.

\begin{theorem}\label{thm:seven-generic-family}
Let \(t\geq0\). Then
\[
\begin{array}{c|ccc}
 t&\dim Q_{e_t}(P(4))&\dim Q_{e_t}(B(4))&\dim\Crel_{e_t}(4)\\ \hline
 0&14&2&0\\
 t\geq1&70&4&0 .
\end{array}
\]
If \(t\geq1\), then
\begin{equation}\label{eq:seven-family-cohit-basis}
 \bigl\{[m_{a^{[t-1]}}]:a\in\mathscr A_{10}\bigr\}
\end{equation}
is a basis of \(Q_{e_t}(B(4))\).
If a monomial of degree \(e_t\) is ordinary hit, then
its orbit sum is symmetrically hit.
\end{theorem}
\begin{proof}
Propositions~\ref{prop:small-initial-degrees} and
\ref{prop:degree-ten-basis} give the initial degrees three
and ten, including their cohit bases and zero relative dimensions.
With \(e=3\), Lemma~\ref{lem:kameko-binary-criterion} has
\[
 \alpha(3+4)=3,\qquad \alpha(3+3)+t-1=t+1 .
\]
Thus \(\mu(e_t)=4\) for \(t\geq1\) precisely when \(t\geq2\),
and \(s_4(3)=1\). Propositions~\ref{prop:kameko-cohit} and
\ref{prop:kameko-relative} identify the ordinary, symmetric,
and relative spaces at every index \(t\geq1\) with their
degree-ten spaces. This proves the dimensions, with the ordinary values agreeing
with Sum \cite[Theorem~4]{Sum2013} and ordinary Kameko reduction.
Proposition~\ref{prop:transport-complete-bases} proves
\eqref{eq:seven-family-cohit-basis}.
The relative space is zero in every degree of the family, so
Corollary~\ref{cor:relative-detection} proves the monomial implication.
\end{proof}

For \(t\geq1\), the four tuples in
\eqref{eq:seven-family-cohit-basis} are
\[
\begin{aligned}
 &(2^{t+1}-1,\ 2^{t+1}-1,\
                    3\cdot2^{t-1}-1,\ 3\cdot2^{t-1}-1),\\
 &(2^{t+1}-1,\ 2^{t+1}-1,\ 2^{t+1}-1,\ 2^t-1),\\
 &(2^{t+2}-1,\ 2^t-1,\ 2^t-1,\ 2^t-1),\\
 &(2^{t+2}-1,\ 2^{t+1}-1,\ 2^{t-1}-1,\ 2^{t-1}-1).
\end{aligned}
\]
The entire symmetric hit basis is determined at the same time.
For the first family, apply
Proposition~\ref{prop:transport-complete-bases} with
\(E=16\) and \(u=t-2\); for the second, use \(E=10\) and
\(u=t-1\). In either case, every orbit sum outside
\(\mathscr R_{E,u}^{(4)}\), together with every transported
vector of the finite symmetric hit basis, is included exactly
once. Thus the construction specifies the complete hit space
as well as its quotient. In the stable ranges the ranks are
\begin{equation}\label{eq:generic-hit-ranks}
\begin{aligned}
 r_{d_t}&=N_{d_t}-73,&
 s_{d_t}&=b_{d_t}-6,&
 a_{d_t}&=N_{d_t}-68 \qquad(t\geq2),\\
 r_{e_t}&=N_{e_t}-70,&
 s_{e_t}&=b_{e_t}-4,&
 a_{e_t}&=N_{e_t}-66 \qquad(t\geq1).
\end{aligned}
\end{equation}
These formulas follow from Proposition~\ref{prop:rank-formula}
and the two theorems, including inclusion ranks five and four,
respectively.

\subsection{Adjacent vanishing families and the degree-eighteen action}

The same degree reduction applies to the neighboring initial
degrees that occur in the fixed-preimage constructions. The
finite values supporting those applications and the first
relative class are recorded separately from the two principal
families.

\begin{table}[H]
\centering\small
\caption{Auxiliary exact ranks. The ordinary dimensions
\(N_d-r_d\) agree with \cite[Theorem~4]{Sum2013} and
ordinary Kameko reduction.}
\label{tab:auxiliary-ranks}
\begin{tabular}{r|rrrrrr}
\toprule
\(d\)&\(N_d\)&\(b_d\)&\(r_d\)&\(s_d\)&\(a_d\)&\(c_d\)\\
\midrule
0&1&1&0&0&1&0\\
1&4&1&0&0&1&0\\
2&10&2&4&1&5&0\\
3&20&3&6&1&8&0\\
4&35&5&14&3&16&0\\
5&56&6&41&5&42&0\\
6&84&9&60&7&62&0\\
7&120&11&85&8&88&0\\
8&165&15&110&11&114&0\\
9&220&18&174&15&177&0\\
10&286&23&216&19&220&0\\
11&364&27&300&25&302&0\\
12&455&34&434&32&436&0\\
13&560&39&525&38&526&0\\
14&680&47&630&43&634&0\\
15&816&54&741&49&745&1\\
18&1330&84&1204&76&1210&2\\
22&2300&136&2184&129&2191&0\\
23&2600&150&2445&146&2449&0\\
25&3276&185&3156&183&3158&0\\
\bottomrule
\end{tabular}
\end{table}

\begin{proposition}\label{prop:auxiliary-initial-degrees}
Let \(0\leq d\leq15\), or let \(d\in\{18,22,23,25\}\).
Then the ranks are those in Table~\ref{tab:auxiliary-ranks} below.
If \(d=11\), then \([m_{3,3,3,2}]\) and \([m_{7,3,1,0}]\)
form a basis of \(Q_{11}(B(4))\).
If \(d\in\{22,23,25\}\), then the orbit sums indexed by the
corresponding row of Table~\ref{tab:adjacent-bases} represent
a basis of \(Q_d(B(4))\).
\end{proposition}
\begin{proof}
The matrices are formed by \eqref{eq:cartan} and
\eqref{eq:coefficient-test}, with power-of-two square indices
as justified by Lemma~\ref{lem:power-generators}. Their exact
reductions give the \(r_d,s_d,a_d\) columns in
Table~\ref{tab:auxiliary-ranks}. The \(c_d\) column is obtained
by substituting these values in \eqref{eq:rank-formula}.
For degree zero the source matrices are empty and inclusion
is the identity. Appendix~\ref{app:finite-verifier} reconstructs every entry
from integer binomial coefficients, checks all three ranks, and
prints the nonpivot symmetric cohit representatives.
The elimination invariant proved there gives both spanning and
independence. In degrees \(11,22,23,25\), these representatives
are exactly the tuples stated here. The ordinary dimensions
\(N_d-r_d\) also agree with Sum \cite[Theorem~4]{Sum2013},
using ordinary Kameko reduction when \(\mu(d)=4\).
In particular, the values at \(0\leq d<15\) are zero, while
\(c_{15}=54-49-745+741=1\). This identifies the first degree
of relative nonvanishing used in Theorem~\ref{thm:fifteen}.
\end{proof}

\begin{corollary}\label{cor:degree-families}
Let \(22\leq e\leq26\) and \(t\geq0\). Then
\(\Crel_{\Delta_t(e)}(4)=0\).
The symmetric cohit dimensions are \(7,4,4,2,2\) for
\(e=22,23,24,25,26\), respectively. A basis is given by
\([m_{a^{[t]}}]\), where \(a\) runs through the row for
\(e\) in Table~\ref{tab:adjacent-bases}.
If a monomial of degree \(\Delta_t(e)\) is ordinary hit,
then its orbit sum is symmetrically hit.

\begin{table}[H]
\centering\small
\caption{Initial partitions for the adjacent vanishing families.}
\label{tab:adjacent-bases}
\begin{tabular}{c|l}
\toprule
Initial degree&Partitions representing a symmetric cohit basis\\
\midrule
22&\((7,5,5,5),(7,7,4,4),(7,7,5,3),(7,7,7,1)\)\\
  &\((15,3,2,2),(15,3,3,1),(15,7,0,0)\)\\
23&\((7,7,5,4),(7,7,7,2),(15,3,3,2),(15,7,1,0)\)\\
24&\((7,7,5,5),(7,7,7,3),(15,3,3,3),(15,7,1,1)\)\\
25&\((7,7,7,4),(15,7,3,0)\)\\
26&\((7,7,7,5),(15,7,3,1)\)\\
\bottomrule
\end{tabular}
\end{table}
\end{corollary}
\begin{proof}
Proposition~\ref{prop:auxiliary-initial-degrees} gives the initial
bases, dimensions, and zero relative spaces in degrees \(22,23,25\).
The degree \(24=e_2\) case follows from
Theorem~\ref{thm:seven-generic-family}.
In degree \(26=2\cdot11+4\),
\(\alpha(26+4)=4\leq4\) and \(\alpha(26+2)=3>2\).
Equation~\eqref{eq:mu} and parity therefore give \(\mu(26)=4\).
Propositions~\ref{prop:kameko-cohit} and
\ref{prop:kameko-relative} transport the degree-eleven basis
and its zero relative space to degree twenty-six. The two
transported tuples are exactly those in
Table~\ref{tab:adjacent-bases}.

For \(e=22,23,24,25,26\), the values of
\(\alpha(e+4)\) are \(3,4,3,4,4\), and the values of
\(\alpha(e+3)\) are \(3,3,4,3,4\).
Corollary~\ref{lem:kameko-numerical} therefore applies at
every subsequent degree change. The relative isomorphisms
give zero relative spaces, and
Proposition~\ref{prop:transport-complete-bases} gives the
displayed bases and dimensions. The monomial implication follows from
Corollary~\ref{cor:relative-detection}.
\end{proof}

\begin{corollary}\label{cor:nonzero-families}
Let \(t\geq0\). Then
\[
 \Crel_{22\cdot2^t-4}(4)\cong\F C_2
\]
as modules for the normalizer quotient.
\end{corollary}
\begin{proof}
For initial degree \(18\), the two relevant digit sums are
\(\alpha(22)=3\) and \(\alpha(21)=3\).
Theorem~\ref{thm:binary-classification} therefore gives an
equivariant relative isomorphism with degree eighteen.
Theorem~\ref{thm:normalizer-eighteen} identifies that initial
module with \(\F C_2\), proving the assertion.
\end{proof}

\appendix
\section{Exact verification of the finite calculations}
\label{app:finite-verifier}

The finite arguments above are reproducible without any external data
files. We give the complete reference calculation here and its
degree-seventy continuation in Appendix~\ref{app:fixed-verifier}.
The two code blocks, in their displayed order, form one Python program;
Python~3.10 or later suffices. A vector is encoded by an integer whose
binary digits are its coefficients. The function \texttt{increments}
uses integer binomial coefficients to list every nonzero Cartan term,
and \texttt{square} adds each increment to its source tuple.
The ordinary target basis is listed by \texttt{E}; the symmetric
target basis \texttt{P} consists of its nonincreasing tuples.

The elimination has a simple exact invariant. Each stored column has
a distinct largest nonzero coordinate. Reducing a new column changes
it only by a linear combination of stored columns. If its remainder is
zero, it belongs to their span. Otherwise its largest nonzero
coordinate is new, and the remainder is independent of the stored
columns. Induction proves that the stored columns form a basis of
the span of all processed columns. Thus their number establishes
both bounds for the rank. The unit vectors at the other coordinates
give a quotient basis: elimination proves spanning, and a nonzero
linear combination of stored columns has largest coordinate at one
of their pivots, proving independence of the other coordinate classes.

In the symmetric calculation, only the coefficient at a single
nonincreasing representative is retained after expanding all distinct
source permutations. This is exactly
\eqref{eq:coefficient-test}. The program reconstructs the ranks in
Tables~\ref{tab:finite-initial-ranks} and~\ref{tab:auxiliary-ranks},
prints the cohit bases, and checks the monomial implication in each
listed degree. It also verifies ordinary hit membership of \(u,v\),
annihilation of all symmetric hit columns by \(\Theta_u,\Theta_v\),
and their evaluations on the normalizer images.

\medskip

\begin{lstlisting}[language=Python, frame=single, basicstyle=\footnotesize\ttfamily, breaklines=true]
from functools import lru_cache
from itertools import permutations
from math import comb

@lru_cache(None)
def E(d, n=4):
    if n == 1: return ((d,),)
    return tuple((a,) + b for a in range(d+1) for b in E(d-a, n-1))
@lru_cache(None)
def P(d): return tuple(a for a in E(d) if a == tuple(sorted(a, reverse=True)))
@lru_cache(None)
def orbit(a): return tuple(sorted(set(permutations(a))))
@lru_cache(None)
def increments(a, r):
    if not a: return ((),) if r == 0 else ()
    return tuple((i,) + u for i in range(min(r,a[0])+1)
                 if comb(a[0],i) % 2
                 for u in increments(a[1:],r-i))
@lru_cache(None)
def square(a, r):
    return tuple(tuple(x+y for x,y in zip(a,u))
                 for u in increments(a,r))
def xor(values):
    v = 0
    for w in values: v ^= w
    return v

def remainder(v, B):
    while v and v.bit_length()-1 in B: v ^= B[v.bit_length()-1]
    return v

def extend(B, columns):
    B = B.copy()
    for v in columns:
        v = remainder(v, B)
        if v: B[v.bit_length()-1] = v
    return B

def calculation(d):
    es, ps = E(d), P(d)
    I, J = {a:i for i,a in enumerate(es)}, {a:i for i,a in enumerate(ps)}
    rs = [1 << j for j in range(d.bit_length()) if 2*(1 << j) <= d]
    O = [xor(1 << I[b] for b in square(a,r)) for r in rs for a in E(d-r)]
    S = [xor(1 << J[c] for b in orbit(a) for c in square(b,r) if c in J)
         for r in rs for a in P(d-r)]
    A = [xor(1 << I[b] for b in orbit(a)) for a in ps]
    OB, SB = extend({}, O), extend({}, S)
    ranks = (len(OB), len(SB), len(extend(OB, A)))
    for a in ps:
        if not remainder(1 << I[a], OB): assert not remainder(1 << J[a], SB)
    basis = tuple(a for i,a in enumerate(ps) if i not in SB)
    return ranks, basis, (es, ps, I, J, OB, SB, S, A)

expected = [(0,0,1),(0,0,1),(4,1,5),(6,1,8),(14,3,16),(41,5,42),
            (60,7,62),(85,8,88),(110,11,114),(174,15,177),(216,19,220),
            (300,25,302),(434,32,436),(525,38,526),(630,43,634),
            (741,49,745),(896,58,901),(1053,67,1057),(1204,76,1210)]
expected = dict(enumerate(expected))
expected.update({22:(2184,129,2191),23:(2445,146,2449),24:(2855,165,2859),
                 25:(3156,183,3158),26:(3590,204,3592)})
for d, target in expected.items():
    ranks, basis, data = calculation(d)
    assert ranks == target, (d, ranks, target)
    print(d, len(E(d)), len(P(d)), *ranks, 'basis', basis)
    if d == 18: data18 = data

es, ps, I, J, OB, SB, S, A = data18
u = A[J[(7,7,4,0)]]
v = xor(A[J[a]] for a in ((5,5,5,3),(7,4,4,3),(7,7,2,2)))
assert remainder(u, OB) == remainder(v, OB) == 0

def tau_monomial(a):
    terms = {(0,0,0,0)}
    for i, exponent in enumerate(a):
        for k in range(exponent.bit_length()):
            if exponent & (1 << k):
                new = set()
                for b in terms:
                    for j in range(4):
                        if i != j:
                            c = tuple(b[h] + ((1 << k) if h == j else 0)
                                      for h in range(4))
                            if c in new: new.remove(c)
                            else: new.add(c)
                terms = new
    return terms

def tau(f):
    return xor(1 << I[b] for a in es if f >> I[a] & 1
               for b in tau_monomial(a))
Tu = ((7,7,4,0),(8,7,3,0),(9,6,3,0),(9,7,2,0),(10,5,3,0),
      (10,7,1,0),(11,4,3,0),(11,5,2,0),(11,6,1,0),(12,3,3,0))
Tv = ((7,7,2,2),(9,7,1,1),(11,5,1,1),(13,3,1,1))
for terms in (Tu, Tv):
    functional = sum(1 << J[a] for a in terms)
    assert all((functional & c).bit_count() % 2 == 0 for c in S)
images = (tau(u), tau(v))
pairing = [[sum(f >> I[a] & 1 for a in terms) % 2 for f in images]
           for terms in (Tu,Tv)]
assert pairing == [[1,1],[0,1]]
print('All finite rank, monomial, membership, and normalizer checks passed.')
\end{lstlisting}

\section{Exact fixed-preimage construction in degree seventy}
\label{app:fixed-verifier}

The degree-seventy assertion requires invariance under
\(H=\langle(12),(34)\rangle\). We now retain source expressions
during the same elimination. The target basis is formed by the
\(H\)-orbit sums with all four exponents positive. Restriction to
this subspace is valid because Steenrod squares preserve the exact
support of every nonzero monomial term, as proved in
Proposition~\ref{prop:support}.

The continuation below lists source tuples in increasing lexicographic
order and adjoins the square blocks in the order \(1,2,4,8,16\).
For each column, \texttt{rep} records a bit vector in the ordered list
\texttt{sources}. Every elimination step adds the same stored source
expression that is added to the target column. Inductively, each
stored vector is therefore exactly the indicated sum of Cartan
columns. Reducing either target to zero returns an explicit expression
using \(249\) or \(251\) fixed source orbit sums. The program prints
both complete lists. Its last check expands these sources in the full
ordered monomial basis and compares the resulting support with the
singleton target. It therefore verifies the fixed-preimage identities
independently of the orbit-coordinate membership test.

\medskip

\begin{lstlisting}[language=Python, frame=single, basicstyle=\footnotesize\ttfamily, breaklines=true]
square.cache_clear()
increments.cache_clear()

def paired(d):
    return tuple(a for a in E(d) if min(a) > 0
                 and a[0] >= a[1] and a[2] >= a[3])

def pair_orbit(a):
    x,y,z,w = a
    return tuple(sorted({(x,y,z,w),(y,x,z,w),
                         (x,y,w,z),(y,x,w,z)}))

rows = paired(70)
index = {a:i for i,a in enumerate(rows)}
assert len(rows) == 13685
B, R, sources = {}, {}, []
stages = {1:(13090,6681),2:(12529,11666),4:(11440,13306),
          8:(9455,13571),16:(6201,13632)}
for r, (count, rank) in stages.items():
    ps = paired(70-r)
    assert len(ps) == count
    for a in ps:
        v = xor(1 << index[c] for b in pair_orbit(a)
                for c in square(b,r) if c in index)
        rep = 1 << len(sources)
        sources.append((r,a))
        while v and v.bit_length()-1 in B:
            k = v.bit_length()-1
            v ^= B[k]
            rep ^= R[k]
        if v:
            k = v.bit_length()-1
            B[k], R[k] = v, rep
    assert len(B) == rank
    print('Fixed sources:', r, count, rank)

for name, target, count in [('A',(21,21,14,14),249),
                            ('B',(22,22,13,13),251)]:
    v, rep = 1 << index[target], 0
    while v:
        k = v.bit_length()-1
        assert k in B
        v ^= B[k]
        rep ^= R[k]
    selected = [a for j,a in enumerate(sources) if rep >> j & 1]
    assert len(selected) == count
    output = set()
    for r,a in selected:
        for b in pair_orbit(a):
            for c in square(b,r):
                if c in output: output.remove(c)
                else: output.add(c)
    assert output == {target}
    print(name, 'fixed-preimage sources:', selected)
print('Both degree-seventy fixed-preimage identities verified.')
\end{lstlisting}

\medskip
\newpage
\subsection*{Funding}
\DJ{}\d{\u{a}}ng V\~o Ph\'uc was funded by the Post-Doctoral Scholarship Programme of the Vingroup Innovation Foundation (VINIF), Institute of Big Data, under project code VINIF.2024.STS.38.

\subsection*{Conflict of interest}
The authors declare no conflict of interest.

\end{document}